\documentclass[11pt, a4paper]{amsart}
\usepackage{amssymb,latexsym,amsmath, graphicx} 
\usepackage{enumerate,enumitem,color,hyperref} 
\usepackage{mathrsfs}
\usepackage{calligra}
\usepackage[T1]{fontenc}

\input{xypic}
\xyoption{all}

\title[\bf Arithmetic of intertwining operators]{An arithmetic property of intertwining operators \\ for $p$-adic groups}

\author{ \bf A. Raghuram} 

\date{\today}      

\address{A. Raghuram: Department of Mathematics\\ Indian Institute of Science Education and Research, Dr. Homi Bhabha Road, Pashan, Pune Maharashtra 411008, INDIA.} 

\email{raghuram@iiserpune.ac.in}

\numberwithin{equation}{subsection}   
\newtheorem{lemma}[equation]{Lemma}

\newtheorem{thm}[equation]{Theorem}
\newtheorem{prop}[equation]{Proposition}
\newtheorem{cor}[equation]{Corollary}

\newtheorem{defn}[equation]{Definition}

\newtheorem{exam}[equation]{Example}
\newtheorem{hyp}[equation]{Hypothesis}

\usepackage{bm}
\usepackage{upgreek}
\makeatletter
\newcommand{\bfgreek}[1]{\bm{\@nameuse{up#1}}}

\let\oldtocsection=\tocsection
\let\oldtocsubsection=\tocsubsection
\let\oldtocsubsubsection=\tocsubsubsection
\renewcommand{\tocsection}[2]{\hspace{0em}\oldtocsection{#1}{#2}}
\renewcommand{\tocsubsection}[2]{\hspace{1em}\oldtocsubsection{#1}{#2}}
\renewcommand{\tocsubsubsection}[2]{\hspace{2em}\oldtocsubsubsection{#1}{#2}}

\begin{document}

 


\maketitle

 \def\R{\mathbb{R}}
\def\C{\mathbb{C}}
\def\Z{\mathbb{Z}}
\def\Q{\mathbb{Q}}
\def\A{\mathbb{A}}
\def\F{\mathbb{F}}
 \def\J{\mathbb{J}}
\newcommand\D{{\mathbb{ D}}}
 \def\L{\mathbb{L}}
\def\bG{\mathbb{G}}
\def\N{\mathbb{N}}
\def\BH{\mathbb{H}}
\newcommand\Qi{\mathbb{Q}({\bf i})}

\def\bfG{\mathbf{G}}
\def\bfB{\mathbf{B}}
\def\bfT{\mathbf{T}}
\def\bfU{\mathbf{U}}
\def\bfP{\mathbf{P}}
\def\bfM{\mathbf{M}}
\def\bfN{\mathbf{N}}
\def\bfA{\mathbf{A}}

\def\LN{{}^L\!N}
\def\LP{{}^L\!P}
\def\LM{{}^L\!M}
\def\LA{{}^L\!A}
\def\LG{{}^L\!G}
\def\Ln{{}^L\mathfrak{n}}
\def\Lg{{}^L\mathfrak{g}}

\newcommand{\vless}{\rotatebox[origin=c]{-90}{$<$}}
\newcommand{\vgreat}{\rotatebox[origin=c]{90}{$<$}}
\newcommand{\vgreater}{\rotatebox[origin=c]{90}{$\leq$}}

\newcommand\Dm{\D_\lambda}
\newcommand\Dmp{\D_{\lambda^\prime}}
 \newcommand\Dim{\D_{\io\lambda}}
 \newcommand\Dimp{\D_{\io\lambda^\prime}}
\newcommand\Dum{\D_{\ul{\lambda}}}
\newcommand\Dium{\D_{\io\ul{\lambda}}} 
\newcommand\DumN{\D_{\ul{\lambda}-N\gamma_P}} 
\newcommand\DiumN{\D_{\io\ul{\lambda}-N\gamma_P}} 
\newcommand\cal{\mathcal}
\newcommand\SMK{{\cal S}^M_{K^M_f}}
\newcommand\tMZl{\tM_{\lambda,\Z}} 
\newcommand\Gm{{\mathbb G}_m}
\newcommand\cA{\cal A}
\newcommand\cC{\cal C}
\newcommand\calL{\cal L}
\newcommand\cO{\cal O}
\newcommand\cU{\cal U}
\newcommand\cK{\cal K}   
\newcommand\cW{\cal W}     
\newcommand\HH{{\cal H}}
\newcommand\cF{\mathcal{F}} 
\newcommand\G{\mathcal{G}}
\newcommand\cB{\mathcal{B}}
\newcommand\cT{\mathcal{T}}
\newcommand\cS{\mathcal{S}}
\newcommand\cP{\mathcal{P}}
\newcommand\Exp{\mathcal{E}xp}

\newcommand\GL{{ \rm  GL}}
\newcommand\Gl{{ \rm  GL}}
\newcommand\U{{ \rm  U}}
\def\SU{{\rm SU}}
\def\S{{\bf S}}
\newcommand\Gsp{{\rm Gsp}}
\newcommand\Lie {{ \rm Lie}} 
\newcommand\Sl{{ \rm  SL}}
\newcommand\SL{{ \rm  SL}}
\newcommand\SO{{ \rm  SO}}
\newcommand\rO{{\rm  O}}
\newcommand\GU{{\rm  GU}}
\newcommand{\Sp}{{\rm Sp}}
\newcommand\Ad{{\rm Ad}}
\newcommand\Sym{{\rm Sym}}

\def\ringO{\mathcal{O}}
\def\idealP{\mathfrak{P}} 
\def\g{\mathfrak{g}}
\def\a{\mathfrak{a}}
\def\k{\mathfrak{k}}
\def\z{\mathfrak{z}}
\def\s{\mathfrak{s}}
\def\c{\mathfrak{c}}
\def\b{\mathfrak{b}}
\def\t{\mathfrak{t}}
\def\q{\mathfrak{q}}
\def\l{\mathfrak{l}}
\def\gl{\mathfrak{gl}}
\def\sl{\mathfrak{sl}}
\def\u{\mathfrak{u}} 
\def\fp{\mathfrak{p}} 
\def\p{\mathfrak{p}}   
\def\r{\mathfrak{r}}
\def\fd{\mathfrak{d}}
\def\fR{\mathfrak{R}}
\def\fI{\mathfrak{I}}
\def\fJ{\mathfrak{J}}
\def\i{\mathfrak{i}}
\def\perm{\mathfrak{S}}
\newcommand\fg{\mathfrak g}
\newcommand\fk{\mathfrak k}
\newcommand\fgK{(\mathfrak{g},K_\infty^0)}
\newcommand\gK{ \mathfrak{g},K_\infty^0 }

\newcommand\ul{\underline} 
\newcommand\tp{{  {\pi}_f}}
\newcommand\tv{ {\pi}_v}
\newcommand\ts{{ {\pi}_f}}
\newcommand\pts{ {\pi}^\prime_f}
\newcommand\usf{\ul{\pi}_f}
\newcommand\pusf{\ul{\pi}^\prime_f}
\newcommand\usvp{\ul{\pi}^\prime_v}
\newcommand\usv{\ul{\pi}_v}
\newcommand\io{{}^\iota}
\newcommand\uls{{\underline{\pi}}}

\newcommand\Spec{\hbox{\rm Spec}} 
\newcommand\SGK{\mathcal{S}^G_{K_f}}
\newcommand\SMP{\mathcal{S}^{M_P}}
 \newcommand\SGn{\mathcal{S}^{G_n}}
 \newcommand\SGp{\mathcal{S}^{G_{n^\prime}}}
\newcommand\SMPK{\mathcal{S}^{M_P}_{K_f^{M_P}}}
\newcommand\SMQ{\mathcal{S}^{M_Q}}
\newcommand\SMp{\mathcal{S}^{M }_{K_f^M}}
\newcommand\SMq{\mathcal{S}^{M^\prime}_{K_f^{M^\prime}}}
\newcommand\uSMP{\ul{\mathcal{S}}^{M_P}}
\newcommand\SGnK{\mathcal{S}^{G_n}_{K_f}}
\newcommand\SG{\mathcal{S}^G}
\newcommand\SGKp{\mathcal{S}^G_{K^\prime_f}}
\newcommand\piKK{{ \pi_{K_f^\prime,K_f}}}
\newcommand\piKKpkt{\pi^{\pkt}_{K_f^\prime,K_f}}
\newcommand\BSC{ \bar{\mathcal{S}}^G_{K_f}}
\newcommand\PBSC{\partial\SGK}
\newcommand\pBSC{\partial\SG}
\newcommand\PPBSC{\partial_P\SGK}
\newcommand\PQBSC{\partial_Q\SGK}
\newcommand\ppBSC{\partial_P\mathcal{S}^G}
\newcommand\pqBSC{\partial_Q\mathcal{S}^G}
\newcommand\prBSC{\partial_R\mathcal{S}^G}
\newcommand \bs{\backslash} 
 \newcommand \tr{\hbox{\rm tr}}
 \newcommand\ord{\text{ord}}
\newcommand \Tr{\hbox{\rm Tr}}
\newcommand\HK{\mathcal{H}^G_{K_f}}
\newcommand\HKS{\mathcal{H}^G_{K_f,\place}}
\newcommand\HKv{\mathcal{H}^G_{K_v}}
\newcommand\HGS{\mathcal{H}^{G,\place}}
\newcommand\HKp{\mathcal{H}^G_{K_p}}
\newcommand\HKpo{\mathcal{H}^G_{K_p^0}}
\newcommand\ch{{\bf ch}}

\newcommand\M{\mathcal{M}}
\newcommand\Ml{\M_\lambda}
\newcommand\tMl{\tilde{\Ml}}
\newcommand\tM{\widetilde{\mathcal{M}}}
\newcommand\tMZ{\tM_\Z}
\newcommand\tsigma{\ul{\pi}}
\newcommand \pkt{\bullet}
\newcommand\tH{\widetilde{\mathcal{H}}}
\newcommand\Mot{{\bf M}} 
\newcommand\eff{{\rm eff}}
\newcommand\Aql{A_{\q}(\lambda)}
\newcommand\wl{w\cdot\lambda}
\newcommand\wlp{w^\prime\cdot\lambda} 

\def\w{{\bf w}} 
\def\d{{\sf d}}
\def\e{{\bf e}} 
\def\x{{\tt x}}
\def\y{{\tt y}}
\def\v{{\sf v}}
\def\q{{\sf q}} 
\def\ff{{\bf f}}
\def\bk{{\bf k}}
 
\def\Ext{{\rm Ext}}
\def\Aut{{\rm Aut}}
\def\Hom{{\rm Hom}}
\def\Ind{{\rm Ind}}
\def\Asai{{\rm Asai}}
\def\aInd{{}^{\rm a}{\rm Ind}}
\def\aIndPG{\aInd_{\pi_0(P(\R)) \times P(\A_f)}^{\pi_0(G(\R)) \times G(\A_f)}}
\def\aIndQG{\aInd_{\pi_0(Q(\R)) \times Q(\A_f)}^{\pi_0(G(\R)) \times G(\A_f)}}
\def\Gal{{\rm Gal}}
\def\End{{\rm End}} 
\newcommand\Coh{{\rm Coh}}  
\newcommand\Eis{{\rm Eis}}
\newcommand\Res{\mathrm{Res}}
\newcommand\place{\mathsf{S}}
\newcommand\emb{\mathcal{I}} 
\newcommand\LB{\mathcal{L}}  
\def\Hod{{\mathcal{H}od}}
\def\Crit{{\rm Crit}}
  
 \newcommand\ip{\pi_f\circ \iota} 
\newcommand\Wp{W_{\pi_\infty\times \ip}}
\newcommand\Wpc{W^{\text{cusp}}_{\pi_\infty\times \pi_f\circ\iota}}
\newcommand\Lcusp{  L^2_{\text{cusp}}(G(\Q)\bs G(\A_f)/K_f)}
\newcommand\MiC{\tM_{\iota\circ\lambda,\C} }
\newcommand\miC{\M_{\iota\circ\lambda,\C} }
\newcommand\tpl{{}^\iota}
\newcommand\Id{\rm Id}
\newcommand\Lr{L^{\text{rat}}}
\newcommand \iso{ \buildrel \sim \over\longrightarrow} 
\newcommand\us{\ul\pi}
\newcommand\qvs{q_v^{-z}}
\newcommand \into{\hookrightarrow}
\newcommand\ppfeil[1]{\buildrel #1\over \longrightarrow}
\newcommand\eb{{}^\iota}
    
\def\bfpi{\mathbf{\Pi}}
\def\bfdelta{\mathbf{\Delta}}

\def\sI{\mathscr{I}}
\def\sU{\mathscr{U}}
\def\sJ{\mathscr{J}}

\bigskip
\tableofcontents 
 
\section{Introduction}
\label{sec:intro}

If one proposes to use the theory of Eisenstein cohomology to prove algebraicity results for the special values of automorphic $L$-functions 
as in my work with Harder \cite{harder-raghuram-book}, or its generalizations in my recent papers: \cite{raghuram}, with Bhagwat \cite{bhagwat-raghuram}, 
and with Krishnamurthy \cite{muthu-raghuram},  
then in a key step, one needs to prove that the normalised standard intertwining operator between induced representations for 
$p$-adic groups has a certain arithmetic property. 
The principal aim of this article is to address this particular local problem in the generality of the Langlands--Shahidi machinery; the main result of this 
article is invoked in \cite{bhagwat-raghuram} and \cite{muthu-raghuram}, and I expect that it will be useful in future investigations on the arithmetic properties of automorphic $L$-functions. 

\medskip

Let $F$ be a $p$-adic field, that is, a non-archimedean local field of characteristic $0$ with finite residue field $k_F$. 
Let $\bfG$ be a connected reductive group defined over $F$; assume that 
$\bfG$ is quasi-split over $F$. Fix a choice of Borel subgroup $\bfB$ of $\bfG$ defined over $F.$  Write $\bfB = \bfT \bfU$, where 
$\bfT$ is a maximal torus, and $\bfU$ the unipotent radical of $\bfB$; both defined over $F.$ 
Suppose $\bfP$ is a maximal parabolic subgroup of $\bfG$ defined over $F$, assumed to be standard, i.e., 
containing $\bfB$, and with Levi decomposition $\bfP= \bfM \bfN.$ Let $\bfA$ denote the maximal central split torus of $\bfM$. 
The $F$-points of $\bfG, \bfB, \bfT, \bfU, \bfP, \bfM,$ $\bfN,$ and $\bfA$ are denoted by $G, B, T, U, P, M, N$, and $A$, respectively.  
To emphasise the dependence on $P$, we also denote $M = M_P$, $N = N_P,$ and $A = A_P.$ 
Let $\pi$ be an irreducible admissible representation of $M_P$. (In the applications we have in mind, $\pi$ will be a local component of a global cuspidal automorphic 
representation of cohomological type.) Let $I_P^G(s, \pi)$ be the induced representation as in the Langlands--Shahidi theory \cite{shahidi-book};  
the precise definitions are recalled in the main body of this article; for the introduction suffice it to say that it is obtained by normalised parabolic induction from 
$P$ to $G$ of $\pi$ with the complex variable $s$ introduced in a delicate manner. Let ${\bf \Pi}$ be the set of simple roots of $G$ with respect to $B$ and $\alpha_P$ 
the unique simple root corresponding to $P$, and $w_0$ the unique element in the Weyl group of $G$ such that $w_0({\bf \Pi} \setminus \{\alpha_P\}) 
\subset {\bf \Pi}$ and 
$w_0(\alpha_P) < 0.$ Let $Q$ be the parabolic subgroup of $G$ associate to $P$, and ${}^{w_0}\pi$ the corresponding representation 
of $M_Q$. We denote 
$
T_{\rm st}(s, \pi) : I_P^G(s, \pi) \to I_Q^G(-s, {}^{w_0}\pi)
$ 
for the standard intertwining operator. There is a choice of measure implicit in the integral that defines the intertwining operator. 
Consider the Langlands dual groups: 
let ${}^LP^\circ = {}^LM^\circ {}^LN^\circ$ be the Levi decomposition of the parabolic subgroup ${}^LP^\circ$ of ${}^LG^\circ$ corresponding to $P$. We write 
${}^L\mathfrak{n} = \oplus_{j=1}^m r_j$ for the decomposition of ${}^L\mathfrak{n}$ under the adjoint action of ${}^LM^\circ$; it is a multiplicity free direct sum. 
Given $\pi$ and $r_j$, the local aspects of the Langlands--Shahidi machinery  
attach a local $L$-factor $L(s, \pi, \tilde{r}_j)$ when $\pi$ is generic, i.e., admits a Whittaker model. 
Denote by $k$ the {\it point of evaluation}, which, by definition, is the point such that 
$$I_P^G(s,\pi)|_{s=k} \ = \ \aInd_P^G(\pi),$$ 
where the right hand side is the algebraic (un-normalized) parabolic induction of $\pi$ to a representation of $G$; see Def.\,\ref{def:point-eval}. 
For brevity, let 
$\fI = I_P^G(k, \pi) = \aInd_P^G(\pi)$ and $\tilde\fI = I_Q^G(-k, {}^{w_0}\pi).$ 

\medskip
Now we impose an arithmetic context: suppose $E$ is a `large enough' finite Galois extension of $\Q$, and suppose there is a smooth absolutely-irreducible 
admissible representation $(\sigma, V_{\sigma, E})$ of $M_P$ on an $E$-vector space $V_{\sigma, E}$ such that for some embedding 
$\iota : E \to \C$ we have ${}^\iota\sigma \cong \pi$. The induced modules $\fI_0 = \aInd_P^G(\sigma)$ and 
$\tilde{\fI}_0 = I_Q^G(-k, {}^{w_0}\sigma)$ give $E$-structures on $\fI$ and $\tilde{\fI}$, i.e., the canonical
$\fI_0 \otimes_{E, \iota} \C \to \fI$ and $\tilde{\fI}_0 \otimes_{E, \iota} \C \to \tilde{\fI}$ are isomorphisms. 
For the parabolic subgroup $P$, assume (i) the local Langlands correspondence to be known for $M_P$; this is a serious condition which is met in a lot of examples, and widely expected to hold in all generality with prescribed desiderata, and (ii) that $P$ satisfies an integrality property: $\rho_P|_{A_P} \in X^*(A_P)$ -- see 
Sect.\,\ref{sec:framework} for notations not defined in the introduction.  
For the representation ${}^\iota\sigma$, motivated by global considerations, assume 
(i) ${}^\iota\sigma$ to be unitary up to a half-integral Tate twist, 
(ii) ${}^\iota\sigma$ to be essentially tempered, 
(iii) the point of evaluation $s=k$ to be `on the right of the unitary axis' (Def.\,\ref{def:right-of-u-axis}) that guarantees absolute convergence of the integral defining 
the standard intertwining operator at $s = k$; and 
(iv) ${}^\iota\sigma$ is generic. 
The first main result (Thm.\,\ref{thm:waldy-arithmetic}) of this article is an arithmeticity result for the standard intertwining operator at the point of evaluation $s = k$, i.e., 
there is an $E$-linear $G$-equivariant map $T_{\rm arith} : \fI_0 \to \tilde{\fI}_0$ such that 
$$
T_{\rm st}(s, {}^\iota\sigma)|_{s = k} = T_{\rm arith} \otimes_{E,\iota} \C. 
$$
The proof involves keeping track of arithmeticity in the proof of a rationality result for the standard intertwining operator for $p$-adic groups due to Waldspurger \cite[Thm.\,IV.I.I]{waldspurger}. 

\medskip
The integrality property on $P$ seems to tie up remarkably with motivic considerations; this is already very interesting 
in the example (see Sec.\,\ref{sec:example-rankin-selberg}) of $G = \GL(N)$ and $P$ maximal such that $M_P = \GL(n) \times \GL(n')$  in which case this integrality
translates to $nn' \equiv 0 \pmod{2}$ which is exactly the condition in \cite{harder-raghuram-book} imposed therein due to motivic considerations.

\medskip
Consider the normalised standard intertwining operator defined as:
\begin{equation}
\label{eqn:T-norm}
T_{\rm norm}(s, \pi) \ = \ \left(\prod_{j=1}^m \frac{L(js, \pi, \tilde r_j)}{L(js+1, \pi, \tilde r_j)}\right)^{-1} T_{\rm st}(s, \pi), \quad \Re(s) \gg 0.
\end{equation}
Continuing with all the hypotheses as above, at the point of evaluation $s=k$, the local $L$-values $L(jk, {}^\iota\sigma, \tilde r_j)$ and $L(jk+1, {}^\iota\sigma, \tilde r_j)$ are finite; and hence $T_{\rm norm}(s, \pi)|_{s = k}$ is convergent. 
We impose a `criticality' condition on $s = k$ that imposes a half-integrality property on $k$ and is entirely a function of the parabolic subgroup $P$ and 
the ambient group $G$. This condition has a global motivation in that the corresponding global $L$-values at $s = k$ are critical values 
in the sense of Deligne \cite{deligne}, and like the integrality condition on $P$ it restricts the scope of global applications; see, for example, the interesting case of $G = \Sp(2n)$ and $M_P = \GL(n)$  in Sec.\,\ref{sec:example-exterior-square} that involves the exterior square $L$-functions for $\GL(n)$. The arithmeticity result on local critical $L$-values for Rankin--Selberg $L$-functions \cite[Prop.\,3.17]{raghuram-imrn} 
gives the impetus to hypothesize that 
$$
L(s_j, {}^\iota\sigma, \, \tilde r_j) \ \in \ \iota(E), \quad s_j  \in \{jk, \, jk+1\},
$$
and furthermore this $L$-value is Galois equivariant; see Hyp.\,\ref{hyp:local-L-value}. 
Under such a hypothesis, which can be verified in many examples such as when $G$ is a classical group, 
the main result in 
Thm.\,\ref{thm:waldy-arithmetic} can be strengthened to Thm.\,\ref{thm:waldy-arithmetic-normalised} that gives an arithmeticity result for the normalised standard intertwining operator at the point of evaluation.

\medskip 
The results of this article (Thm.\,\ref{thm:waldy-arithmetic} and Thm.\,\ref{thm:waldy-arithmetic-normalised}) say that if we use Eisenstein cohomology to give a cohomological interpretation of Langlands's constant term theorem, and so attempt to prove a rationality result for ratios of critical values of automorphic $L$-functions, then at any given finite place we do not pick up any 
possibly-transcendental period. 
Suppose $\pi$ is an unramified representation, i.e., has a vector fixed under the hyper-special maximal compact subgroup of $M_P$, then both 
$\fI$ and $\tilde\fI$ are also unramified; suppose $f_0 \in \fI$ (resp., $\tilde{f}_0 \in \tilde\fI$) is the normalised spherical vector; then Langlands's generalization of the classical Gindikin--Karpelevic formula says that $T_{\rm norm}(f_0) = \tilde{f}_0.$ This implies the theorem because the $E$-structures are generated by 
these normalised spherical vectors. The real content of the theorem is that it works for any $\pi$ whether or not it is unramified.  
Whereas the global theory of Eisenstein cohomology and the special values of automorphic $L$-functions provides the context, this article is purely local ($p$-adic) in nature, and does not need the reader to be familiar with such global aspects.

\bigskip

{\small
{\it Acknowledgements:} 
I thank G\"unter Harder and Freydoon Shahidi for conversations on arithmetic properties that find their way into this article. I also thank Chandrasheel Bhagwat who acted as a sounding board while I was finalising the manuscript. 
I acknowledge support from a MATRICS research grant MTR/2018/000918 from the Science and Engineering Research Board, Department of Science and Technology, Government of India.}

\bigskip
\section{Local aspects of the Langlands--Shahidi machinery}
\label{sec:framework}

\medskip
\subsection{Induced representations and `the point of evaluation'}
\label{sec:ind-rep-pt-eval}
Let $\delta_P$ be the modulus character of $P$; it is trivial on $N_P$ and its values on $M_P$ is given by: 
$$
\delta_P(m) \ = \ |\det(\Ad_{N_P}(m))|, \quad m \in M_P, 
$$
where $\Ad_{N_P} : M_P \to \GL(\Lie(N_P))$ is the adjoint representation of $M_P$ on the Lie algebra of $N_P$, and $|\ |$ is the normalised absolute value on $F$. 
Let $Z(M_P)$ be the centre of $M_P$, $A_P$ the maximal split torus in $Z(M_P)$. 
Let $X^*(A_P) = \Hom(A_P, F^*)$ and $X^*(M_P) = \Hom(M_P, F^*)$ denote the group of rational characters of $A_P$ and $M_P$. 
Restriction from $M_P$ to $A_P$ gives an inclusion
$X^*(M_P) \hookrightarrow X^*(A_P)$, which induces an isomorphism $X^*(M_P) \otimes_\Z \Q \cong X^*(A_P) \otimes_\Z \Q.$ 
The modulus character $\delta_P$ is naturally an element of $\a_P^* := X^*(A_P) \otimes_\Z \R.$ Fix a Weyl group invariant 
inner product $(\ , \ )$ on $X^*(A_P) \otimes_\Z \R.$ 

\medskip
Let ${\bf \Delta}_G$ be the set of all roots which are naturally in $X^*(T)$;  
for the choice of the Borel subgroup $B$, let ${\bf \Delta}_G^+$ be the set of positive roots, and ${\bf \Pi}_G$ the set of simple roots. 
Let $\rho_P$ be half the sum of all positive roots whose root spaces are in 
$\Lie(N_P)$; via the restriction from $T$ to $A_P$ we have $\rho_P \in X^*(A_P) \otimes_\Z \Q.$ We have the equality: 
$$
|2 \rho_P(m)| = \delta_P(m), \quad \forall m \in M.
$$

\medskip

Let $\a_P = \Hom(X^*(A_P), \R) = \Hom(X^*(M_P), \R)$ denote the real Lie algebra of $A_P,$ and $H_P : M_P \to \a_P$ denote the Harish-Chandra homomorphism defined by: 
$$
\exp\langle \chi, H_P(m)\rangle \ := \ |\chi(m)|, \quad \forall \chi \in X^*(A_P), \ \forall m \in M.
$$
In particular, taking $\chi = \rho_P$, we get: 
$$
\exp\langle \rho_P, H_P(m) \rangle \ = \ \delta_P(m)^{1/2}, \quad \forall m \in M.
$$

\medskip
Let $(\pi, V_\pi)$ be an irreducible admissible representation of $M_P;$ the representation space $V_\pi$ is a vector space over $\C$.  
For $\nu \in \a_P^* \otimes_\R \C$, define the induced representation 
$$
I(\nu,\pi) \ := \ \Ind_P^G(\pi \otimes \exp(\langle \nu, H_P(\ )\rangle) \otimes 1_U), 
$$
where $\Ind$ means normalised parabolic induction. 
The representation space $V(\nu, \pi)$ is the vector space of all smooth (i.e., locally constant) functions $f : G \to V_\pi$ such that 
$$
f(mng) \ = \ \pi(m) \exp(\langle \nu + \rho_P, H_P(m)\rangle ) f(g), \quad \forall g \in G, \ m\in M, \ n\in N.
$$
Recall that $P$ is a maximal parabolic subgroup, defined by a simple root $\alpha_P$ which is the unique simple root whose root space is in 
$\Lie(N_P).$  
Set $\langle \rho_P, \alpha_P \rangle = 2\frac{(\rho_P, \alpha_P)}{(\alpha_P, \alpha_P)}$ and put 
$$
\gamma_P \ := \ \tilde \alpha_P \ := \ \frac{1}{\langle \rho_P, \alpha_P \rangle} \, \rho_P.
$$
In the Langlands--Shahidi machinery the notation $\tilde \alpha$ is commonly used; elsewhere in the arithmetic theory of automorphic forms 
the notation $\gamma_P$ is commonly used; it is the fundamental weight corresponding to the simple root $\alpha_P.$
For $s \in \C$, define $\nu_s$ as:
$$
\nu_s \ := \ s \tilde\alpha_P \ = \ \frac{s}{\langle \rho_P, \alpha_P \rangle} \, \rho_P.
$$
Let $I(s, \pi) := I(s\tilde\alpha_P, \pi)$, whose representation space $V(s, \pi) := V(s\tilde\alpha_P, \pi)$ consists of all locally constant 
functions $f : G \to V_\pi$ such that 
$$
f(mng) \ = \ \pi(m) \, \delta_P(m)^{\frac12 + \frac{s}{2 \langle \rho_P, \alpha_P \rangle}} \, f(g), \quad \forall g \in G, \ m\in M, \ n\in N.
$$

\begin{defn}[Point of evaluation]
\label{def:point-eval}
Define the point of evaluation $k$ as 
$$
k \ := \ - \langle \rho_P, \alpha_P \rangle, 
$$
which depends only on $P$ and $G$, and has the property that $I(k, \pi) = \aInd_P^G(\pi)$ which is the algebraic (i.e., un-normalised) parabolic induction 
from $P$ to $G$ of the representation $\pi.$ 
\end{defn}
The point of evaluation $k$ is half-integral, i.e., $k \in \Z$ or $k \in \tfrac12 + \Z$, or more succinctly $2k \in \Z$, since 
$\langle \beta, \alpha_P \rangle = 2(\beta, \alpha_P)/(\alpha_P, \alpha_P) \in \Z$ for any root $\beta$. In general, $k$ can be integral or a genuine half-integer; for example, if $G = \GL(N)$ with $N \geq 2$, and $P$ is any maximal parabolic subgroup, then $k = -N/2$; see Sec.\,\ref{sec:example-rankin-selberg}.

\medskip
\subsection{The standard intertwining operator: definition and analytic properties}
\label{sec:standard-int-op}
For the maximal parabolic subgroup $P = P_{\Theta}$, where $\Theta = {\bf \Pi} \setminus \{\alpha_P\}$, recall that $w_0$ is the unique element in the Weyl group 
such that $w_0(\Theta) \subset {\bf \Pi}$ and $w_0(\alpha_P) < 0;$ let $Q = P_{w_0(\Theta)}$ be the maximal parabolic subgroup associate to $P$. Then 
$M_Q = w_0 M_P w_0^{-1},$ and let  ${}^{w_0}\pi$ be the representation of $M_Q$ given by conjugation. 

Let $f \in I_P^G(s, \pi)$ and $g \in G$. Suppose there exists a vector $v$ in the inducing representation of $I_Q^G(-s, {}^{w_0}\pi)$, 
such that for all $\check{v}$ in the contragredient of this inducing representation 
 the integral $\int_{N_Q} \langle f(w_0^{-1} n g), \check{v} \rangle dn$ 
converges absolutely to $\langle v, \check{v} \rangle$ 
then define $\int_{N_Q} \langle f(w_0^{-1} n g) \, dn = v.$ If this is verified for all $f \in I_P^G(s, \pi)$ and all $g \in G$ then 
define an intertwining operator for $G$-modules
$$
T_{\rm st}(s) : \ I_P^G(s, \pi) \ \longrightarrow \ I_Q^G(-s, {}^{w_0}\pi)
$$ 
 by the integral
\begin{equation}
\label{eqn:intertwining-operator-shahidi}
T_{\rm st}(s)(f)(g) \ = \ \int_{N_Q} f(w_0^{-1} n g) \, dn; \quad f \in V(s, \pi), \ g \in G. 
\end{equation}
Assume, here and henceforth, that the measures in such intertwining integrals are chosen to be $\Q$-valued. 
The operator $T_{\rm st}(s)$ is denoted as $A(s, \pi, w_0)$ in Shahidi \cite[Sec.\,4.1]{shahidi-book} (see also Kim \cite[Sec.\,4.3]{kim-notes}). 
That $T_{\rm st}(s)(f) \in I_Q^G(-s, {}^{w_0}\pi)$ is verified in {\it loc.\,cit.} 
The 
following convergence statement is a special case of \cite[Prop.\,4.1.2]{shahidi-book}:  

\begin{prop}
\label{prop:abs-conv}
If $\Re(s) \gg 0$ then $T_{\rm st}(s)(f)(g)$ converges absolutely for all $g \in G$ and all $f \in V(s, \pi).$
\end{prop}

If $T_{\rm st}(s)(f)(g)$ converges absolutely for all $g \in G$ and all $f \in V(s, \pi),$ then we will simply say that 
$T_{\rm st}(s)$ converges absolutely. One can be more specific about the domain of convergence in the tempered case: 

\begin{prop}
\label{prop:abs-conv-tempered}
If $\pi$ is a tempered (unitary) representation, then the standard intertwining operator $T_{\rm st}(s)$ converges absolutely for $\Re(s) > 0.$ 
\end{prop}

The above convergence statements are contained in Harish-Chandra's work on harmonic analysis on $p$-adic reductive groups; 
the reader is referred to Shahidi \cite[Sect.\,2.2]{shahidi-ajm81} and the references therein; see also Kim \cite[Prop.\,12.3]{kim-notes}.

\medskip
Without worrying about convergence, let us see the shape of the standard intertwining operator at the point of evaluation $s = k.$ The domain 
of $T_{\rm st}(s)|_{s = k},$ as noted above, is $I(k, \pi) = \aInd_P^G(\pi)$. The codomain is 
$I_Q^G(-s, {}^{w_0}\pi) := I(-s\tilde\alpha_Q, {}^{w_0}\pi)$, whose representation space consists of all locally constant 
functions $f' : G \to V_{{}^{w_0}\pi} = V_\pi$ such that 
$$
f'(m'n'g') \ = \ {}^{w_0}\pi(m') \, |\delta_Q(m')|^{\frac12 - \frac{s}{2 \langle \rho_Q, \alpha_Q \rangle}} \, f'(g'), \quad \forall g' \in G, \ m'\in M_Q, \ n'\in N_Q,
$$
where ${}^{w_0}\pi(m') = \pi(w_0^{-1}m' w_0).$ 
Put $s = k = - \langle \rho_P, \alpha_P \rangle$; since $\langle \rho_Q, \alpha_Q \rangle = \langle \rho_P, \alpha_P \rangle$ we get:
$$
f'(m'u'g') \ = \ {}^{w_0}\pi(m') \, \delta_Q(m') \, f'(g').
$$
Hence, at the point of evaluation, in terms of un-normalised induction we get: 
\begin{equation}
T_{\rm st}(s)|_{s = k} : \ \aInd_P^G(\pi) \ \longrightarrow \ \aInd_Q^G({}^{w_0}\pi \otimes \delta_Q).
\end{equation}

\medskip
\subsection{Local factors and the local Langlands correspondence}
\label{sec:LLC}

A defining aspect of the Langlands program is Langlands's computation (\cite[Sect.\,5]{langlands-euler}) 
of the constant term of an Eisenstein series, which at a local unramified place boils down 
to computing the standard intertwining operator on `the' spherical vector which is a scalar multiple of the spherical vector on the other side, and this scalar multiple is 
an expression denoted $M(s)$ in {\it loc.\,cit.} Langlands says that J.\,Tits pointed out to him how to express $M(s)$ in a more 
convenient form. This is now an important ingredient in the Langlands--Shahidi machinery; see Shahidi \cite[Sect.\,2]{shahidi-annals88}.   

\medskip

Let $\LG^\circ$ be the complex reductive group which is the connected component of the Langlands dual $\LG$ of $G$; see Borel \cite[I.2]{borel-corvallis}; 
and let $\LP$ be the parabolic subgroup of $\LG$ corresponding to $P$, and 
$\LN$ its unipotent radical. 
The Levi quotient $\LM^\circ$ of $\LP^\circ$ acts on the Lie algebra $\Ln$ of $\LN^\circ$ by the adjoint action. There is a positive integer $m$ such that the set 
$\{\langle \tilde\alpha_P, \beta \rangle\}$ -- as $\beta$ varies over positive roots such that the root space $\Lg(\beta^\vee)$ of the dual root $\beta^\vee$ is in 
$\Ln$ -- is $\{1,\dots,m\}.$ For each $1 \leq j \leq m$ put 
\begin{equation}
\label{eqn:V-j}
V_j \ = \ \mbox{span of $\Lg(\beta^\vee)$ for $\beta$ such that $\langle \tilde\alpha_P, \beta \rangle = j.$}
\end{equation}
Then the action of $\LM^\circ$ on $\Ln$ stabilizes each $V_j$ and furthermore acts irreducibly on $V_j$. Denote $r_j$ the action of $\LM^\circ$ on $V_j$, and 
$\Ln = \oplus_{j=1}^m r_j$ is a multiplicity free decomposition as an $\LM^\circ$-representation. Let 
$\tilde r_j$ denote the contragredient of $r_j.$

\medskip

Given a smooth irreducible admissible representation $\pi$ that is {\it generic}, i.e., has a Whittaker model, and for $1 \leq j \leq m$, 
the local aspects of the Langlands--Shahidi machinery attaches a local $L$-factor $L(s, \pi, \tilde r_j)$ (see Shahidi~\cite{shahidi-annals90}) 
which is the inverse of a polynomial in $q^{-s}$ of degree at most $d_j := \dim(V_j);$ when $\pi$ is unramified this degree is $d_j$.

\medskip

Let $W_F$ be the Weil group of $F$, and $W_F' = W_F \times \SL_2(\C)$ the Weil--Deligne group. The local Langlands correspondence for $G$ says that 
to $\pi$ corresponds its Langlands parameter which is an admissible homomorphism 
$\phi_\pi : W_F' \to \LM_P$; see Borel~\cite[Sect.\,8]{borel-corvallis} for the requirements on the parameter $\phi_\pi.$ 
Composing with $\tilde r_j$ gives $\tilde r_j \circ \phi_\pi : W_F' \to {}^L \GL_{d_j},$ an admissible homomorphism which  
parametrizes, via the local Lanlgands correspondence for $\GL_{d_j}(F)$, a smooth irreducible admissible representation 
of  $\GL_{d_j}(F)$ that we denote $\tilde r_j(\pi).$ As in Shahidi \cite[p.\,3]{shahidi-kyoto} we will impose the working hypothesis: 
\begin{equation}
\label{eqn:LLC-L-factors}
L(s, \pi, \tilde r_j) \ = \ L(s, \tilde r_j(\pi)), 
\end{equation}
that is known in a number of instances; see the references in {\it loc.\,cit.}

\medskip
\subsection{The notion of being on the right of the unitary axis}
\label{sec:right-of-u-axis} 

Recall that $\pi$ is a smooth irreducible admissible representation of $M_P$, which is to be a local component 
of a globally generic cuspidal automorphic representation (needed by the context in which we can evoke the Langlands--Shahidi machinery), and 
keeping the generalised Ramanujan conjecture in the back of our minds (see, for example, Shahidi~\cite{shahidi-asian}), we will impose the condition that $\pi$ is essentially tempered, i.e., tempered mod the centre. 

\medskip
Let $M_P^1 = \bigcap_{\chi \in X^*(M_P)} {\rm Ker}(|\chi|_F),$ the subgroup of $M_P$ generated by all compact subgroups. For the split centre of $M_P$, 
say $A_P \cong F^* \times \cdots \times F^*$, 
and $\eta_i : A_P \to F^*$ is the projection to the $i$-th copy. Define $A_P^1 \subset A_P$ similar to $M_P^1$; we have $A_P^1 = A_P \cap M_P^1.$ 
Let $X(M_P) = \Hom(M_P/M_P^1, \C^*)$; similarly, $X(A_P)$. Restricting from $M_P$ to $A_P$ gives an isomorphism $X(M_P) \cong X(A_P).$ 
Given $\ul z = (z_1,\dots, z_l) \in \C^l$, we get 
an unramified character $A_P \to \C^*$ given by $|\eta_1|^{z_1} \otimes \cdots \otimes  |\eta_l|^{z_l}$, and via $X(M_P) \cong X(A_P)$ an unramified character of 
$M_P$ which we will denote as $\ul\eta^{\ul z}$.
Given $\pi$ as above, tempered modulo the centre means that there exist exponents 
$\ul e = (e_1, \dots, e_l) \in \R^l$ and a smooth irreducible unitary tempered representation $\pi^t$ such that 
$\pi \ \cong \ \pi^t \otimes \ul\eta^{\ul e}.$
(Keeping global applications in mind, we will impose later a hypothesis that the exponents $e_i$ are (half-)integral.) The representation $\tilde r_j(\pi)$ of $\GL_{d_j}(F)$, obtained by functoriality, is also tempered modulo its centre. A few words of explanation might be helpful. There is an exponent
$\tilde f_j = f(\tilde r_j, e_1, \dots, e_l) \in \R$, that depends on the representation $r_j$ and the exponents $e_1,\dots,e_l$, such that 
$$
\tilde r_j(\pi) \ \cong \ \tilde r_j(\pi)^t \otimes |\det |^{\tilde f_j}, 
$$
with $\tilde r_j(\pi)^t = \tilde r_j(\pi^t)$ being a unitary tempered irreducible representation of $\GL_{d_j}(F)$. Local functoriality preserves temperedness, by
the desiderata in Borel \cite[10.4, (4)]{borel-corvallis}, and the one has to keep track of the central characters for which consider the diagram: 
\begin{equation}
\label{eqn:diagram-central}
\xymatrix{
W_F' \ar[r]^{\phi_\pi} & \LM_P^\circ \ar[rr]^{\! \! \! \! \! \! \! \! \!  \! \! \! \! \! \! \tilde r_j} & & {}^L(\GL_{d_j})^\circ = \GL_{d_j}(\C) \\
 & \LA_P^\circ \ar@{^{(}->}[u] \ar[rr] & & {}^L(Z(\GL_{d_j}))^\circ = \C^* \ar@{^{(}->}[u] 
}
\end{equation}
Since the representation $\tilde r_j$ is irreducible, the centre $\LA_P^\circ \cong (\C^*)^l$ of $\LM_P^\circ$ acts via scalars, explaining the bottom horizontal arrow. 
The unramified character $\ul\eta^{\ul e}$ of $A_P$ corresponds to its Satake parameter $\vartheta_{{\ul \eta}^{\ul e}}$
in $\LA_P^\circ$; we get $\tilde r_j(\vartheta_{{\ul \eta}^{\ul e}}) \in \C^*$, which corresponds to an unramified character $|\ |^{\tilde f_j}$ of $F^*$, or 
the character $|\det |^{\tilde f_j}$ of $\GL_{d_j}(F)$,  for some exponent 
$\tilde f_j$ which, {\it a priori}, lives in $\C$, but since $e_j \in \R$ and $\tilde r_j$ is an algebraic representation, it is clear that $\tilde f_j \in \R.$ 

\begin{lemma}
\label{lem:f_j=jf_1}
With notations as above we have $\tilde f_j = j \cdot \tilde f_1.$
\end{lemma}

\begin{proof}
The proof follows from the definition of $V_j$ in \eqref{eqn:V-j} which is the representation space for $r_j$. 
(It is instructive to see this detail in the example discussed in Sec.\,\ref{sec:example-exterior-square}.)
\end{proof}

\medskip
\begin{defn}
\label{def:right-of-u-axis}
Let $\pi$ be a smooth irreducible admissible generic representation of $M_P$ that is tempered modulo the centre with exponents 
$e_1,\dots, e_l \in \R$. We say $\pi$ is on the right 
of the unitary axis with respect to the ambient group $G$, if 
$$
- \langle \rho_P, \alpha_P \rangle + \tilde f_1  > 0.
$$
\end{defn}

By Lem.\,\ref{lem:f_j=jf_1} it follows that for each $1 \leq j \leq m$ we have:
$- j \langle \rho_P, \alpha_P \rangle + \tilde f_j  > 0.$

\medskip
\begin{cor}
Let $\pi$ be a smooth irreducible admissible generic representation of $M_P$ that is tempered modulo the centre with exponents 
$e_1,\dots, e_l \in \R$, and which is on the right 
of the unitary axis with respect to the ambient group $G$, then the local $L$-values  
$$
L(jk, \pi, \tilde r_j) \ \ {\rm and} \ \ L(jk+1, \pi, \tilde r_j)
$$
are finite for each $1 \leq j \leq m$, where, recall that $k = - \langle \rho_P, \alpha_P \rangle$ is the point of evaluation. 
\end{cor}

\begin{proof}
If $\pi$ is a unitary tempered representation of $\GL_d(F)$ then the standard local $L$-factor $L(s, \pi)$ is finite if $\Re(s) > 0;$ this follows from 
Jacquet's classification of tempered representations of $\GL_d(F)$ and the well-known inductive recipe for local $L$-factors that is succinctly summarised
in Kudla \cite{kudla}. The proof follows from the equalities: 
$L(jk, \pi, \tilde r_j) = L(jk, \tilde r_j(\pi)) = L(jk + \tilde f_j, \tilde r_j(\pi)^t).$
\end{proof}

The condition of being on the right of the unitary axis is tailor-made to appeal to Shahidi's tempered $L$-functions conjecture that is now a theorem after the work of many authors (see \cite[p.\,147]{shahidi-book}) culminating in Heiermann--Opdam \cite{heiermann-opdam}.

\bigskip
\section{An arithmetic variation on a rationality result of Waldspurger}
\label{sec:var-on-waldy}

In this section we recall a rationality result of Waldspurger \cite[Thm.\,IV.1.1]{waldspurger}, and show how to reformulate it so that the statement works at an arithmetic level. Before that let us clarify some terminology that apparently causes some confusion. 

\subsection{Digression on the adjectives: rationality, algebraicity, and arithmeticity}
First of all, even among experts, there seems to be no universal agreement on the precise meaning of these adjectives. In this article, all three words are used, and it might help the reader to clarify their meanings. The word rationality has two meanings and the context usually makes it clear. First of all, a result of the form 
$``({\rm L-value})/({\rm periods}) \in \Q"$ is often called 
a {\it rationality result} for $L$-values. Then there is a common abuse of terminology and a result of the form $``({\rm L-value})/({\rm periods}) \in \overline\Q"$ is also called a rationality result. We will refer to the latter as an {\it algebraicity result} for $L$-values. A second usage of rationality, as in the context of Waldspurger's result, 
comes from algebraic geometry and means that some function or operator at hand is a rational function on an algebraic variety. To explain our usage of the word arithmetic, suppose we have an 
$L$-value at hand, which is the value at $s = s_0$ of the $L$-function $L(s, \pi)$ attached to some object $\pi$ defined over $\C$, for example, $\pi$ can a cuspidal automorphic representation. We may set up our context for the object $\pi$ to have 
an {\it arithmetic} origin, i.e., there is an object $\sigma$ defined over some coefficient field $E$, such that for some embedding of fields $\iota : E \to \C$, 
the base-change ${}^\iota\sigma$ of $\sigma$ via $\iota$ is the object $\pi.$ In such a context
a result of the form $``L(s_0, {}^\iota\sigma)/({\rm periods}) \in \iota(E)"$ is given the appellation of an {\it arithmetic result} for $L$-values. 
With this explanation of the words, the purpose of this section is to show that Waldspurger's rationality result for intertwining operator has an arithmetic origin. 
We will use the word {\it arithmeticity} for the behaviour of an arithmetic result upon changing $\iota$, or equivalently, by replacing $\iota$ by $\tau \circ \iota$ 
for any $\tau \in \Gal(\bar\Q/\Q)$; this is compatible with the usage of arithmeticity as in \cite{gan-raghuram}.

\medskip
\subsection{A rationality result of Waldspurger}
\label{sec:waldy-rationality}
In this subsection we will adumbrate the presentation in \cite[IV.1]{waldspurger}. 
Recall the notations $M_P^1$, $A_P^1$, $X(M_P),$ and $X(A_P)$ from Sect.\,\ref{sec:right-of-u-axis}. When $P$ is fixed we drop the subscript $P$ from $M_P$, $A_P$, etc. 

\medskip
Note that $X(M)$ has the structure of an 
algebraic variety over $\C$; denote by $\cB$ the $\C$-algebra of polynomial functions on $X(M)$. Let $(\pi, V)$ be a smooth admissible representation of $M$ on 
a $\C$-vector space $V$, and let 
$\cO_\C = \{\pi \otimes \chi : \chi \in X(M)\}.$ A function $f : \cO_\C \to \C$ is a polynomial if there exists $b \in \cB$ such that 
$f(\pi \otimes \chi) = b(\chi)$. For an open set $\cU \subset \cO_\C$, a function $f:\cU \to \C$ is a rational function if there exists $b_1, b_2 \in \cB$ 
such that $b_1(\chi) f(\pi \otimes \chi) = b_2(\chi)$ for all $\chi \in X(M)$ with $\pi  \otimes \chi \in \cU$ and $b_1(\chi) \neq 0.$ 

\medskip

Let $I_P^G(\pi \otimes \chi)$ be the normalised parabolically induced representation. Restriction from $G$ to its maximal compact subgroup $K$ sets up an isomorphism 
$I_P^G(\pi \otimes \chi) \cong I_{K\cap P}^K(\pi).$ Let $P'$ be a maximal parabolic subgroup of $G$ that has the same Levi subgroup $M = M_{P'} = M_P.$ For each 
$\pi \otimes \chi \in \cO_\C$ suppose we are given a $G$-equivariant operator $A(\pi \otimes \chi) : I_P^G(\pi \otimes \chi) \to I_{P'}^G(\pi \otimes \chi)$ that depends 
only on the equivalence class of $\pi \otimes \chi.$ We say that the operator $A(\pi \otimes \chi)$ is a {\it polynomial} if for all $f \in I_{K\cap P}^K(\pi)$ there exist finitely many $f_1,\dots,f_r \in I_{K\cap P'}^K(\pi)$ and $b_1,\dots, b_r \in \cB$ such that 
$A(\pi \otimes\chi)(f) \ = \ \sum_{i=1}^r b_i(\chi) f_i$ for all $\chi \in X(M).$ Furthermore, we say $A(\pi \otimes \chi)$ is {\it rational} if there exists 
$b \in \cB$, such that for all $f \in I_{K\cap P}^K(\pi)$ there exist finitely many 
$f_1,\dots,f_r \in I_{K\cap P'}^K(\pi)$ and $b_1,\dots, b_r \in \cB$ such that 
\begin{equation}
\label{eqn:rational-operator-1}
b(\chi) A(\pi \otimes\chi)(f) \ = \ \sum_{i=1}^r b_i(\chi) f_i, \quad \mbox{for all $\chi \in X(M)$ with $b(\chi) \neq 0.$}
\end{equation}

\medskip

Rationality of the intertwining operators may be formulated in another way that is used in the proof of \cite[Thm.\,IV.1.1]{waldspurger}, and which will allow us to descend 
the statement and proof to an arithmetic level to give us Thm.\,\ref{thm:waldy-arithmetic} below. For $m \in M$, let $b_m \in B$ be defined as $b_m(\chi) = \chi(m).$ 
Define $V_\cB = V \otimes_\C \cB$ on which $M$ acts as: $\pi_\cB(m)(v \otimes b) = \pi(m)v \otimes b_mb.$ For $\chi \in X(M)$, let $\cB_\chi$ be the maximal ideal
$\{b \in \cB : b(\chi) = 0\}.$ Then the action of $M$ on $V_\cB \otimes \cB/\cB_\chi$ is the representatoin $\pi \otimes \chi$. Similarly,
$I_P^G(V_\cB) = I_P^G(V) \otimes_\C \cB$. Let ${\rm sp}_\chi : \pi_\cB \to \pi_\cB \otimes \cB/\cB_\chi = \pi \otimes \chi$ denote the specialization map; same notation also for 
${\rm sp}_\chi : I_P^G(\pi_\cB) \to I_P^G(\pi \otimes \chi)$. The collection $\{A(\pi \otimes \chi)\}_{\pi \otimes \chi \in \cO_\C}$ of operators is {\it polynomial} if and only 
if there exists a $G$-equivariant homormorphism of $\cB$-modules $A_\cB : I_P^G(\pi_\cB) \to I_{P'}^G(\pi_\cB)$ such that  
the following diagram commutes  
\begin{equation}
\label{eqn:the_map_A_B}
\xymatrix{
I_P^G(\pi_\cB)\ar[rr]^{A_\cB} \ar[d]_{{\rm sp}_\chi} & & I_{P'}^G(\pi_\cB) \ar[d]^{{\rm sp}_\chi} \\ 
I_P^G(\pi \otimes \chi)\ar[rr]^{A(\pi \otimes \chi)} & & I_{P'}^G(\pi \otimes \chi)
}
\end{equation}
for all $\chi$, i.e., 
${\rm sp}_\chi \circ A_\cB = A(\pi \otimes \chi) \circ {\rm sp}_{\chi}.$ Similarly, the collection $\{A(\pi \otimes \chi)\}_{\pi \otimes \chi \in \cO_\C}$ of operators 
is {\it rational} if and only if there exists $A_\cB$ as above and an element $b \in \cB$ such that 
\begin{equation}
\label{eqn:rational-operator-2}
A(\pi \otimes \chi) \circ {\rm sp}_{\chi} \circ (1 \otimes b) \ = \ {\rm sp}_\chi \circ A_\cB.   
\end{equation}
Suppose $\tilde f \in I_P^G(\pi_\cB)$ and $A_\cB(\tilde f) = \sum_i \tilde f_i \otimes b_i$, and if ${\rm sp}_\chi$ maps $\tilde f$ to $f$, and similarly, $\tilde f_i$ to $f_i$, 
then \eqref{eqn:rational-operator-2} becomes $b(\chi) A(\pi \otimes \chi)(f) = \sum_i b_i(\chi) f_i$ as in \eqref{eqn:rational-operator-1}. We may and shall talk 
about the collection $\{A(\pi \otimes \chi)\}_{\chi \in \cU}$ of operators being rational on an open subset $\cU \subset X(M).$

\medskip

Let $N'$ be the unipotent radical of $P'$. 
Let $f \in I_P^G(V)$ and $g \in G$. Suppose there exists $v \in V$ such that for all $\check{v} \in \check{V}$ -- the representation space of the contragredient 
$\check{\pi}$ of $\pi$ -- the integral $\int_{N' \cap N \backslash N'} \langle f(n'g), \check{v} \rangle dn'$ converges absolutely to $\langle v, \check{v} \rangle$ 
then define $\int_{N' \cap N \backslash N'} f(n'g) \, dn' = v.$ If this is verified for all $f \in I_P^G(\pi \otimes \chi)$ and all $g \in G$ then 
define an intertwining operator for $G$-modules 
$J(\pi \otimes \chi) : I_P^G(\pi \otimes \chi) \to I_{P'}^G(\pi \otimes \chi)$ as: 
\begin{equation}
\label{eqn:intertwining-operator-waldy}
J(\pi \otimes \chi)(f)(g) \ = \ \int_{N' \cap N \backslash N'} f(n'g) \, dn'. 
\end{equation}
Note the similarities and differences between the integrals in \eqref{eqn:intertwining-operator-shahidi} and \eqref{eqn:intertwining-operator-waldy}.

\medskip

Let $\Sigma(A_P)$ denote the set of roots of $A_P$ in $\Lie(G)$; identify $\Sigma(A_P)$ with a subset of $\a_M^*$. Denote by $\Sigma(P)$ the 
subset of $\Sigma(A_P)$ of those roots whose root spaces appear in $\Lie(P).$
For $P'$ with the same Levi as $P$, let $\bar{P'}$ denote its opposing parabolic subgroup. (This $\bar{P}'$ is the $Q$ from before.) 
The following theorem is contained in \cite[Thm.\,IV.1.1]{waldspurger}.

\begin{thm}
\label{thm:waldy-rational}
Suppose $\pi$ is an irreducible admissible smooth representation of $M$. Then there is an open cone 
$\cU = \{\chi \in X(M) : \langle \Re(\chi), \alpha \rangle > 0, \ \forall \alpha \in \Sigma(P) \cap \Sigma(\bar{P'})\}$ of $X(M)$ such that 
$J(\pi \otimes \chi)$ is defined by the convergent integral \eqref{eqn:intertwining-operator-waldy} for $\chi \in \cU.$ 
The collection of intertwining operators 
$\{J(\pi \otimes \chi)\}_{\chi \in \cU}$ on this cone is rational. 
\end{thm}

It is the rationality assertion in the above theorem that we are particularly interested in (since convergence in our context is already guaranteed by 
Prop.\,\ref{prop:abs-conv}). We summarise the key steps of its proof, and refer the reader to \cite[IV.1]{waldspurger} for all the details and 
also for some of the notations used below even if not defined here because it would take us too far to systematically define them. 

\medskip
\begin{enumerate}
\item ({\it Reduction step}.)
We need a $G$-equivariant homomorphism of $\cB$-modules
$$
J_\cB : I_P^G(\pi_\cB)\to I_{P'}^G(\pi_\cB)
$$ 
as in \eqref{eqn:the_map_A_B}, 
that satisfies the requirement of \eqref{eqn:rational-operator-2}. 
Frobenius reciprocity for Jacquet modules and parabolic induction gives:
$$
\Hom_{G, \cB}(I_P^G(\pi_\cB), \, I_{P'}^G(\pi_\cB)) = \Hom_{M, \cB}(I_P^G(\pi_\cB)_{P'}, \, \pi_\cB), 
$$
where, $I_P^G(\pi_\cB)_{P'}$ is the Jacquet module of $I_P^G(\pi_\cB)$ with respect to $P'$ on which 
the action of $M$ is the canonical action twisted by $\delta_{P'}^{-1/2}$ to account for normalised parabolic induction. It suffices then to construct 
$$
j_\cB \in \Hom_{M, \cB}(I_P^G(\pi_\cB)_{P'}, \, \pi_\cB)
$$
such that the associated map $J_\cB$ via Frobenius reciprocity satisfies \eqref{eqn:rational-operator-2}. 

\medskip
\item ({\it Exponents in the Jacquet module of an induced representation}.)
By the well-known results of Bernstein and Zelevinskii \cite[2.12]{bernstein-zelevinsky} (see also \cite[I.3]{waldspurger}), 
the Jacquet module 
$I_P^G(\pi_\cB)_{P'}$ is filtered by $(M,\cB)$-submodules $\{\cF_{w,P'}\}_{w \in {}^{P'}W^P},$ indexed by a certain totally ordered set ${}^{P'}W^P$ of representatives in the Weyl group, such that for the successive quotients we have an isomorphism
$$
q_w : \cF_{w,P'}/\cF_{w^+,P'} \to I_{M \cap w\cdot P}^M(w \cdot V_{\cB, M \cap w^{-1}\cdot P'}), 
$$
the right hand side being a parabolically induced module of $M$.  Consider these successive quotients for the action of the split centre $A_P$ of $M$; and let $\Exp_w$ 
be the set of exponents, which are characters $A_P \to \cB^\times$, that appear in the (co-)domain of $q_w.$ 
We may suppose that $1 \in {}^{P'}W^P;$  
the image of $q_1$ is $V_\cB$. For any $w \in {}^{P'}W^P$, if $w \neq 1$ then $\Exp_w \cap \Exp_1 = \emptyset;$ see \cite[p.\,280]{waldspurger}.

\medskip
\item ({\it Killing all sub-quotients except one}.) 
Using the theory of resultants, Waldspurger constructs $R \in \cB[A_P]$ and $b \in \cB$ such that $R$ maps the Jacquet module $I_P^G(\pi_\cB)_{P'}$ into 
$\cF_{1,P'}$ and on each of the generalised eigenspace for $\mu \in \Exp_1$ appearing in $\cF_{1,P'} / \cF_{1^+,P'} \otimes {\rm Frac}(\cB)$ it acts 
as homothety by the element $b.$ The required element $j_\cB$ as in (1) is the composition of $R$ followed by  
$\cF_{1,P'} \to \cF_{1,P'} / \cF_{1^+,P'} \stackrel{q_1}{\longrightarrow} V_\cB.$ 
\end{enumerate}

 \medskip
 \subsection{An arithmetic variant of Thm.\,\ref{thm:waldy-rational}}

Let $E$ be a `large enough' finite Galois extension of $\Q.$ 
The meaning of large enough will be explained in context. Let $X_E(M) = \Hom(M/M^1, E^*)$; similarly, $X_E(A)$. 
Restriction from $M$ to $A$ gives an isomorphism $X_E(M) \cong X_E(A).$ If $A = F^* \times \cdots \times F^*$, $l$-copies, then 
$A/A^1 = \varpi_F^\Z \times \cdots \times \varpi_F^\Z,$ with $\varpi_F^\Z$ being the multiplicative infinite cyclic group generated by the 
uniformizer $\varpi_F.$ Also, $X_E(A) = E^* \times \cdots \times E^*$, where $\ul t = (t_1, \dots, t_l) \in E^* \times \cdots \times E^*$ corresponds 
to the character $\chi_{\ul t}$ that maps $\ul a = (a_1,\dots,a_l) \in A$ to $\prod t_i^{{\rm ord}_F(a_i)}.$ An embedding of fields 
$\iota : E \to \C$, gives a map $\iota_* : X_E(M) \to X_\C(M)$ where $\iota_*\chi = \iota \circ \chi.$ The following diagram might help the reader:
$$
\xymatrix{
X_\C(A)  = \Hom(A/A^1,\C^*) \ar[r]& \C^* \times \cdots \times \C^* & \ar[l]\C/(\Z \cdot \tfrac{2\pi i}{\log(q)}) \times \cdots \times \C/(\Z \cdot \tfrac{2\pi i}{\log(q)}) \\ 
X_E(A)  = \Hom(A/A^1,E^*)  \ar[r] \ar[u]^{\iota_*}& E^* \times \cdots \times E^* \ar[u]^{\iota \times \cdots \times \iota}& 
}
$$
For $\ul s =(s_1,\dots,s_l) \in \C/(\Z \cdot \tfrac{2\pi i}{\log(q)}) \times \cdots \times \C/(\Z \cdot \tfrac{2\pi i}{\log(q)})$ put 
$\ul w := q^{\ul s}$, i.e., $\ul w = (w_1,\dots,w_l) = (q^{s_1},\dots, q^{s_l}) \in \C^* \times \cdots \times \C^*$ that corresponds to the character $\chi_{\ul w}$ of $A$ given by $\ul a \mapsto \prod_i w_i^{{\rm ord}_F(a_i)}.$ 
Note that $X_E(M)$ has the structure of an 
algebraic variety over $E$; denote by $\cB_E(M)$ the $E$-algebra of polynomial functions on $X_E(M)$; then $\cB_E(M) = E[t_1,t_1^{-1},\dots,t_l,t_l^{-1}].$ 
Similarly, $\cB_\C(M) = \C[w_1,w_1^{-1},\dots, w_l,w_l^{-1}] = \C[q^{s_1}, q^{-s_1},\dots, q^{s_l}, q^{-s_l}].$
Base-change via the embedding $\iota$ gives: $\cB_E(M) \otimes_{E,\iota} \C = \cB_\C(M).$ 
To homogenise with the notations of \cite{waldspurger} as used in Sec.\,\ref{sec:waldy-rationality}, 
abbreviate $X_\C(M)$ and $\cB_\C(M)$ as $X(M)$ and $\cB$, respectively.

\medskip

\subsubsection{Hypotheses we impose on a representation in the main result have an arithmetic origin}

Let $(\sigma, V_{\sigma, E})$ be a smooth absolutely irreducible admissible representation of $M$ over an $E$-vector space $V_{\sigma, E}.$ 
For an embedding of fields $\iota : E \to \C$, we have 
the irreducible admissible representation ${}^\iota\sigma$ of $M$ on the $\C$-vector space $V_{{}^\iota\sigma} := V_{\sigma, E} \otimes_{E, \iota} \C$. We may apply 
the considerations of Sect.\,\ref{sec:framework} to $({}^\iota\sigma, V_{{}^\iota\sigma}).$ We explicate below all 
the hypotheses we impose on the representation ${}^\iota\sigma$ in the main result Thm.\,\ref{thm:waldy-arithmetic}; these hypotheses are motivated by our global applications, and are expected to have an arithmetic origin. 

\medskip
The global context of a cohomological cuspidal automorphic representation suggests, via purity considerations, the following hypothesis on $\sigma$. Recall that 
for exponents $\ul e = (e_1,\dots,e_l) \in \R^l$, by $\ul\eta^{\ul e} \in X(A) = X(M)$ defined as: 
$\ul\eta^{\ul e}(\ul a) = \prod_i |a_i|^{e_i}$ for $\ul a = (a_1,\dots,a_l) \in A.$  

\begin{hyp}[Arithmeticity for half-integral unitarity]
\label{hyp:half-int-unitary}
Let $(\sigma, V_{\sigma, E})$ be a smooth absolutely irreducible admissible representation of a reductive $p$-adic group $M$ over an $E$-vector space $V_{\sigma, E}.$  
If for one embedding $\iota : E \to \C$, there exists an $l$-tuple of integers $\ul {\sf w} = ({\sf w}_1,\dots, {\sf w}_l)$ such that 
the representation ${}^\iota\sigma \otimes \ul \eta^{\ul{\sf w}/2}$ is unitary then for every embedding $\iota : E \to \C$, 
the representation ${}^\iota\sigma \otimes \ul \eta^{\ul {\sf w}/2}$ is unitary. 
\end{hyp}

The proof is easy and we leave it to the reader. 
It makes sense to call a $\sigma$ satisfying the above hypothesis as {\it half-integrally unitary}.

\begin{hyp}[Arithmeticity for essential-temperedness]
\label{hyp:arith-temp}
Let $(\sigma, V_{\sigma, E})$ be a smooth, absolutely irreducible, admissible, half-integrally unitary representation of a reductive $p$-adic group $M$ over an $E$-vector space $V_{\sigma, E}.$  
If for one embedding $\iota : E \to \C$ the representation ${}^\iota\sigma$ is essentially tempered, then for every embedding $\iota : E \to \C$ the representation 
${}^\iota\sigma$ is essentially tempered. 
\end{hyp}

\begin{proof}[Proof of Hyp.\,\ref{hyp:arith-temp} for $\GL_n(F)$]
For $\GL_n(F)$ this follows from the considerations in Clozel \cite{clozel} while using Jacquet's classification of tempered representations \cite{jacquet}; 
such a proof is well-known to experts and so we will just sketch the details. The reader is also referred to \cite[Sect.\,9.2]{prasad-raghuram} for a summary 
of the classification of tempered representations that we will use below. For a representation $\pi$ of $\GL_n(F)$ and $t \in \R$, $\pi(t)$ denotes $\pi \otimes |\ |^t$. 

\begin{enumerate}
\item Any tempered representation $\pi$ of $G = \GL_n(F)$ is fully induced from discrete series representations; it is of the form: 
$$
\pi = \Ind_{P_{n_1,\dots,n_r}(F)}^{G}(\pi_1\otimes \cdots \otimes \pi_r), 
$$
where $\pi_i$ is a discrete series representation of $\GL_{n_i}(F)$; $\sum_i n_i = n$; $P_{n_1,\dots,n_r}(F)$ is the parabolic subgroup of $G$ with Levi subgroup 
$\GL_{n_1}(F) \times \cdots \times \GL_{n_r}(F)$. 
\smallskip
\item A discrete series representation $\pi_i$ of $\GL_{n_i}(F)$ is of the form 
$$
\pi_i \ = \ Q(\Delta(\sigma_i, b_i)),
$$
where $n_i = a_ib_i$, with $a_i, b_i \in \Z_{\geq 1}$; 
$\sigma_i$ is a supercuspidal representation of $\GL_{a_i}(F)$ such that $\sigma_i(\frac{b_i-1}{2})$ is unitary, 
and $Q(\Delta(\sigma_i, b_i))$ is the unique irreducible quotient of a parabolically
induced representation: 
$$
\Ind_{P_{b_i,\dots,b_i}(F)}^{\GL_{n_i}(F)}(\sigma_i \otimes \sigma_i(1) \otimes \cdots \otimes \sigma_i(b_i-1))
 \ \twoheadrightarrow \ Q(\Delta(\sigma_i, b_i)). 
$$
\end{enumerate}
In both the steps the parabolic induction used is normalised induction which is not, in general, Galois equivariant. As in Clozel \cite{clozel}, we may force Galois equivariance by 
considering a half-integral Tate twisted version of induction. Using the notations of (1), but letting for the moment $\pi_i$ be any irreducible admissible representation of $\GL_{n_i}(F)$, define: 
$$
{}^{T}\Ind_{P_{n_1,\dots,n_r}(F)}^{G}(\pi_1\otimes \cdots \otimes \pi_r) \ := \ 
\Ind_{P_{n_1,\dots,n_r}(F)}^{G}\left(\pi_1(\frac{1-n_1}{2})\otimes \cdots \otimes \pi_r(\frac{1-n_r}{2})\right)(\frac{n-1}{2}).
$$
Suppose $\tau \in \Aut(\C)$; then one may verify that 
$$
{}^\tau({}^{T}\Ind_{P_{n_1,\dots,n_r}(F)}^{G}(\pi_1\otimes \cdots \otimes \pi_r)) \ = \ {}^{T}\Ind_{P_{n_1,\dots,n_r}(F)}^{G}({}^\tau\pi_1\otimes \cdots \otimes {}^\tau\pi_r).
$$
Define a quadratic character of $F^*$ as $\varepsilon_\tau := (\tau \circ |\ |^{1/2})/|\ |^{1/2};$ it is trivial if and only if $\tau$ fixes $q^{1/2},$ where $q$ is the cardinality of the residue field of $F$. Then: 
$$
{}^\tau\Ind_{P_{n_1,\dots,n_r}(F)}^{G}(\pi_1\otimes \cdots \otimes \pi_r) \ = \ 
\Ind_{P_{n_1,\dots,n_r}(F)}^{G}((\pi_1\otimes \varepsilon_\tau^{n-n_1}) \otimes \cdots \otimes (\pi_r \otimes \varepsilon_\tau^{n-n_r})).
$$ 
Similarly, using  Lem.\,3.2.1 and the few lines following that lemma in \cite{clozel}, we have: 
$$
{}^\tau Q(\Delta(\sigma_i, b_i)) \ = \ Q(\Delta({}^\tau\sigma_i \otimes \varepsilon_\tau^{a_i(b_i-1)}, b_i)). 
$$
Of course each ${}^\tau\sigma_i$ is supercuspidal and so also is any of its quadratic twists; furthermore, 
$({}^\tau\sigma_i \otimes \varepsilon_\tau^{a_i(b_i-1)})(\frac{b_i-1}{2})$
is unitary. Hence the $\tau$-conjugate of the tempered representation $\pi$ is tempered. 

\smallskip
Using the notations in the hypothesis, take $\pi = {}^\iota\sigma$, and if $\iota' : E \to \C$ is any other embedding then take $\tau \in \Aut(\C)$ such that 
$\iota' = \tau \circ \iota$. By assumption $\pi = \pi^t \otimes |\ |^{{\sf w}/2}$ for a unitary tempered representation and an integral exponent ${\sf w}$. Then 
\begin{equation}
\label{eqn:tau-pi}
{}^\tau\pi \ = \ {}^\tau\pi^t \otimes (\tau\circ |\ |^{{\sf w}/2}) \ = \ 
({}^\tau\pi^t \otimes \varepsilon_\tau^{\sf w})  
\otimes |\ |^{{\sf w}/2}; 
\end{equation}
By the above argument ${}^\tau\pi^t$ is tempered, and hence so also is 
${}^\tau\pi^t \otimes \varepsilon_\tau^{\sf w}$. 
\end{proof}

\medskip

\begin{proof}[Remarks on the proof of Hyp.\,\ref{hyp:arith-temp} for classical groups] 
For classical groups, using similar argument as in the case of $\GL_n(F)$, and result for $\GL_n(F)$, 
a proof follows from M\oe glin and Tadic's classification \cite{moeglin-tadic} for discrete series and tempered representations. The proof is tedious. We will sketch the argument for even-orthogonal groups. 

Consider $G = \rO_{2n}(F) = \{g \in \GL_{2n}(F): {}^tg \cdot J \cdot g = J\}$ the split even orthogonal group of rank $n$, 
where $J_{i,j} = \delta(i, 2n-j+1).$ Suppose $n = a_1+\cdots + a_q + n_0$, with $a_1,\dots, a_q \geq 1$ and $n_0 \geq 0,$ 
and  
$P_{(a_1,\dots,a_q; n_0)}$ is the parabolic subgroup of $\rO_{2n}(F)$ with Levi subgroup 
$$
M_{(a_1,\dots,a_q; n_0)} = \GL_{a_1}(F) \times \dots \times \GL_{a_q}(F) \times \rO_{2n_0}(F).
$$ 
Let $\pi_0$ be a discrete series representation of $\rO_{2n_0}(F)$ and $\theta_j$ an essentially discrete series representation of $\GL_{a_j}(F)$. For brevity, let 
$$
\theta_1 \times \cdots \times \theta_q \rtimes \pi_0 \ := \Ind_{P_{(a_1,\dots,a_q; n_0)}}^G(\theta_1 \otimes \cdots \otimes \theta_q \otimes \pi_0).
$$ 
This induced representation is a multiplicity-free direct sum of 
tempered representations (see M\oe glin--Tadic \cite[Thm.\,13.1]{moeglin-tadic} and Atobe--Gan \cite[Desideratum 3.9, (6)]{atobe-gan}). Suppose $\pi$ is one such tempered representation: $\pi \hookrightarrow \theta_1 \times \cdots \times \theta_q \rtimes \pi_0$. 
Suppose $m_{h_1,\dots, h_q,x} \in M$ with $h_j \in \GL_{a_j}(F)$ and $x \in \rO_{2n_0}(F)$, then the absolute-value of the 
determinant of the adjoint action of $m_{h_1,\dots, h_q,x}$ on the Lie algebra of the unipotent radical of $P$ is given by: 
\begin{multline*}
\delta_P(m_{h_1,\dots, h_q,x}) =  
\left(|\det(h_1)|^{2n- 2a_1}|\det(h_2)|^{2n- 2(a_1+a_2)}\dots |\det(h_q)|^{2n_0}\right) \cdot \\
\cdot \left( |\det(h_1)|^{a_1-1} |\det(h_2)|^{a_2-1} \dots |\det(h_q)|^{a_q-1} \right).
\end{multline*}
Using this, for $\tau \in \Aut(\C),$ one may verify that: 
\begin{equation}
\label{eqn:tau-pi-O-2n}
{}^\tau\pi \ \hookrightarrow \ 
{}^\tau(\theta_1 \times \cdots \times \theta_q \rtimes \pi_0) \ = \ 
({}^\tau\theta_1 \otimes \varepsilon_\tau^{a_1-1}) \times \cdots \times ({}^\tau\theta_q \otimes \varepsilon_\tau^{a_q-1}) \rtimes {}^\tau\pi_0. 
\end{equation}
By appealing to the above proof for $\GL_{n_0}(F)$, we know that each ${}^\tau\theta_j$, and so also its quadratic twist  
${}^\tau\theta_1 \otimes \varepsilon_\tau^{a_1-1}$, is an essentially discrete series representation. Hence, proof of arithmeticity for tempered 
representation of $\rO_{2n}(F)$ boils down to proving arithmeticity for discrete series representation of $\rO_{2n_0}(F)$. 

\medskip

For the discrete series representation $\pi_0$ of $\rO_{2n_0}(F)$, there exist $a$ and $n_1$ such that $n_0 = a + n_1$, and 
there exist an essentially discrete series representation $\theta$ of $\GL_a(F)$ and a discrete series representation $\pi_1$ of 
the smaller even-orthogonal group $\rO_{2n_1}(F)$ such that $\pi_0$ is one of two possible subrepresentations of $\theta \rtimes \pi_1$; and 
both these sub-representations are in the discrete series; this being the crux of \cite{moeglin-tadic}. Then, as above
${}^\tau\pi \hookrightarrow {}^\tau\theta \otimes \varepsilon_\tau^{a-1} \otimes {}^\tau\pi_1$. An induction argument (see \cite[p.\,721]{moeglin-tadic}) 
concludes the proof as 
the reduction to a smaller even-orthogonal group ends with the case of $\pi_1$ being a supercuspidal representation (in {\it loc.\,cit.}\ called the weak cuspidal support
of $\pi_0$), and clearly conjugation by 
$\tau$ preserves supercuspidality as it leaves the support of a matrix coefficient unchanged. 
\end{proof}

\medskip
It is an interesting problem to prove this for a general $p$-adic group. Assuming that Hyp.\,\ref{hyp:arith-temp} is true, we can then formulate another hypothesis: 

\medskip
\begin{hyp}[Arithmeticity for being on the right of the unitary axis]
\label{hyp:arith-on-the-right}
Let $(\sigma, V_{\sigma, E})$ be a smooth, absolutely irreducible, admissible, half-integrally unitary, essentially tempered representation of a reductive $p$-adic group $M$ over an $E$-vector space $V_{\sigma, E}.$ 
If for one embedding $\iota : E \to \C$ the representation ${}^\iota\sigma$ is to the right of the unitary axis, then for every embedding $\iota : E \to \C$ the 
representation ${}^\iota\sigma$ is to the right of the unitary axis.  
\end{hyp}

\medskip
For $\GL_n(F)$ this follows from \eqref{eqn:tau-pi} since the exponent for $\pi$ and ${}^\tau\pi$ are equal. Similarly, the above hypothesis will follow from 
Hyp.\,\ref{hyp:arith-temp} that the half-integral exponents $\ul {\sf w}/2$ for 
${}^\iota\sigma$ are independent of $\iota$; in particular the exponent $f_1$ of $\tilde r_1({}^\iota\sigma)$ would be independent of $\iota$.

\medskip
\begin{lemma}[Arithmeticity for genericity]
Let $(\sigma, V_{\sigma, E})$ be a smooth absolutely irreducible admissible representation of a reductive quasi-split $p$-adic group $M$ over an $E$-vector space $V_{\sigma, E}.$ 
If for one embedding $\iota : E \to \C$ the representation ${}^\iota\sigma$ is generic, then for every embedding $\iota : E \to \C$ the representation 
${}^\iota\sigma$ is generic.
\end{lemma}

\begin{proof}
Suppose $\ell : {}^\iota\sigma \to \C$ is a Whittaker functional with respect to a character $\psi : U \to \C^*$ (that is nontrivial on all the root spaces corresponding to simple roots). Given another embedding $\iota' : E \to \C$, there exists $\tau \in \Aut(\C)$ such that $\iota' = \tau \circ \iota.$ Then $\tau \circ \ell$ is a Whittaker functional for ${}^{\iota'}\sigma$ with respect to the character $\tau \circ \psi$ of $U$. 
\end{proof}

After the above hypotheses and lemma, it makes sense to say that a smooth absolutely-irreducible admissible representation $(\sigma, V_{\sigma, E})$ of a reductive $p$-adic group $M$ is half-integrally unitary, essentially tempered, to the right of the unitary axis, or generic, if for some, and hence any, embedding $\iota : E \to \C$ 
the representation ${}^\iota\sigma$ is half-integrally unitary, essentially tempered, to the right of the unitary axis, or generic, respectively. 
The first main theorem of this article is the following result.

\subsubsection{An arithmetic variant of Thm.\,\ref{thm:waldy-rational}}

\begin{thm}
\label{thm:waldy-arithmetic}
Let $P = MN$ be a maximal parabolic subgroup of a connected reductive $p$-adic group $G.$  
Let $(\sigma, V_{\sigma, E})$ be a smooth absolutely-irreducible admissible representation of $M$ over an $E$-vector space $V_{\sigma, E}.$ 
Assume that $E$ is large enough to contain the values of the exponents of $A$ that appear in the Jacquet module of 
$\aInd_P^G(\sigma)$ with respect to the associate parabolic subgroup $Q$. Assuming Hyp.\,\ref{hyp:half-int-unitary}, 
Hyp.\,\ref{hyp:arith-temp}, and Hyp.\,\ref{hyp:arith-on-the-right}, we 
suppose that $\sigma$ is  half-integrally unitary, essentially tempered, to the right of the unitary axis, and generic. 
Suppose that $P$ satisfies the integrality condition: $\rho_P|_{A_P} \in X^*(A_P)$, then so does $Q$ and the modular character $\delta_Q$ takes values in 
$\Q^*.$ There exists an $E$-linear $G$-equivariant map
$$
T_{st, E} : \aInd_P^G(\sigma) \ \longrightarrow \ \aInd_Q^G(\sigma \otimes \delta_Q)
$$ 
such that for any embedding $\iota : E \to \C$ we have:
$$
T_{{\rm st}, E} \otimes_{E,\iota} 1_\C  \ = \ T_{{\rm st}, \iota}, 
$$
where $T_{{\rm st}, \iota} = T_{\rm st}(s, {}^\iota\sigma)|_{s = k} : \aInd_P^G({}^\iota\sigma) \to \aInd_Q^G({}^\iota\sigma \otimes \delta_Q)$ is the standard intertwining operator 
at the point of evaluation. 
\end{thm}

\begin{proof}
Fix an $\iota : E \to \C$. For $\chi \in X(M)$, we have the standard intertwining operator 
$$
T_{\rm st}({}^\iota\sigma, \chi) : I_P^G({}^\iota\sigma \otimes \chi) \to I_Q^G({}^{w_0}({}^\iota\sigma \otimes \chi))
$$
given by an integral where it converges. We will ultimately specialize to the point $\chi_k$ corresponding to the point of evaluation 
$k = -\langle \rho_P, \alpha_P\rangle$; note that $\chi_k = -\rho_P$; at this point our hypothesis guarantees convergence. 
Consider Thm.\,\ref{thm:waldy-rational} 
with the small variation that we take the associate parabolic $Q$ and not $P'$ which required $M_{P'} = M_P;$ for $Q$ we have $M_Q$ is the $w_0$-conjugate of $M_P$. This causes no problem as long as we use the correct integral, i.e., 
we use \eqref{eqn:intertwining-operator-shahidi} instead of \eqref{eqn:intertwining-operator-waldy}. 
From Thm.\,\ref{thm:waldy-rational} we get a $(G,\cB)$-module map 
$$
T_\cB: I_P^G({}^\iota\sigma \otimes_\C \cB) \ \longrightarrow \  I_{Q}^G({}^{w_0\iota}\sigma \otimes_\C \cB) 
$$
that satisfies \eqref{eqn:rational-operator-2} with a homothety element $b \in \cB$. 

\medskip

The main steps in the proof of Thm.\,\ref{thm:waldy-rational} ((i) reduction via Frobenius reciprocity, (ii) Jacquet module calculation, and (iii) construction of 
an $M$-equivariant map using an element $R$ in the group ring of $A$ via the theory of resultants) are all purely algebraic in nature. The same proof, but now working with modules over $E$, gives us an $E$-linear map of $(G, \cB_E)$-modules: 
$$
T_{\cB, E} : I_P^G(\sigma \otimes_E \cB_E)  \ \longrightarrow \   I_{Q}^G({}^{w_0}\sigma \otimes_E \cB_E)
$$ 
with a homothety $b_0 \in \cB_E$ such that for any $\iota : E \to \C$ we have: $T_{\cB, E} \otimes_{E,\iota} \C = T_\cB,$ and $b_0 \otimes_{E,\iota} 1 = b.$
Specialise at the point of evaluation in \eqref{eqn:rational-operator-2}, i.e., take 
$\chi = \chi_k = -\rho_P \in X_E(A)$; hence, $b(\chi_k) = b_0(-\rho_P) \in E^*$; note that we have used $\rho_P|_{A_P}$ is an integral weight. We have: 
$$
\iota(b_0(-\rho_P)) T_{\rm st}({}^\iota\sigma, \chi_k) = {\rm sp}_{\chi_k} \circ (T_{\cB,E} \otimes_{E,\iota} 1_\C).
$$ 
For $\chi_0 \in X_E(A),$ if ${\rm sp}_{\chi_0, E} : \cB_E \to E$ denotes the specialization map at an arithmetic level, then clearly, 
${\rm sp}_{\chi_0, E} \otimes_{E, \iota} 1_\C = {\rm sp}_{\chi_0}$. Hence, 
$
T_{\rm st}({}^\iota\sigma, \chi_k) =  (b_0(\rho_P) {\rm sp}_{\chi_k, E} \circ T_{\cB,E}) \otimes_{E,\iota} 1_\C.
$ 
\end{proof}

\bigskip
\section{Arithmeticity of local critical $L$-values}
\label{sec:arithmetic-L-values}

The purpose of this section is to formulate an arithmeticity hypothesis on local critical $L$-values for automorphic $L$-functions. It is a generalization of \cite[Prop.\,3.17]{raghuram-imrn} which was in the context of Rankin--Selberg $L$-functions and was a crucial ingredient in the proof of the main theorem of that article. Using this 
hypothesis we may strengthen Thm.\,\ref{thm:waldy-arithmetic} to give an arithmeticity result for the normalised standard intertwining operator.

\subsection{Criticality condition on the point of evaluation}
\label{sec:criticality-on-k}

In the context of Rankin--Selberg $L$-functions one takes $M_P = \GL_n \times \GL_{n'}$ as a Levi subgroup of 
an ambient $G = \GL_N,$ where $N = n+ n'$. The integrality condition on $P$ forces $nn'$ to be even. 
For an inducing data $\pi \times \pi'$ of $M_P$, the critical set for the $L$-function $L(s, \pi \times \pi'^{\sf v})$ consists 
of integers if $n \equiv n' \pmod{2}$ and consists of half-integers, i.e., elements of $\tfrac12 + \Z$ if $n \not\equiv n' \pmod{2}$; see \cite[Def.\,7.3]{harder-raghuram-book}. 
The purpose of this subsection is to formalise such parity constraints in the context of Langlands-Shahidi machinery. 

\medskip

Suppose $A_i = \eta_i(A) = F^*$ and $A$ is the internal product $A_1 \times \dots \times A_l$; correspondingly, suppose $M = M_1\cdots M_l$ an almost direct product 
of reductive subgroups, with $A_i$ in the centre of $M_i$. 
Let $\rho_{M_i}$ be half the sum of positive roots for $M_i$. If $\rho_{M_i}$ is integral, then put $\varepsilon_{M_i} = 0.$ 
If $\rho_{M_i}$ is not integral, then necessarily $2 \rho_{M_i}$ is integral, and put $\varepsilon_{M_i} = 1.$ Fix an unramified character 
$\chi_P^{\varepsilon_P/2} \in \Hom(M/M^1, \C^*) = \Hom(A/A^1, \C^*),$ defined by 
$$
\chi_P^{\varepsilon_P/2}(a_1,\dots,a_l) \ := \ |a_1|^{\varepsilon_{M_1}/2} \dots |a_l|^{\varepsilon_{M_l}/2}, \quad  (a_1,\dots,a_l) \in A = F^* \times \cdots \times F^*.
$$ 
Let $\vartheta_P \in \LA_P^\circ$ be the Satake parameter of $\chi_P^{\varepsilon_P/2}.$ Using \eqref{eqn:diagram-central}, there exists $h_j \in \tfrac12 \Z$ such that 
$\tilde r_j(\vartheta_P) = q^{-h_j}$ or that $\tilde r_j(\chi_P^{\varepsilon_P/2}) = |\ |^{h_j}.$ 
Let $\pi$ be an irreducible admissible half-integrally unitary, essentially-tempered, generic representation of $M_P$. 
Consider $\pi \otimes \chi_P^{\varepsilon_P/2}$; we have:  
$$
L(s, \pi, \tilde r_j) \ = \ 
L(s - h_j, \pi \otimes \chi_P^{\varepsilon_P/2}, \tilde r_j).
$$
The idea is that given $\pi,$ we {\it algebrise} it by considering the twist $\pi \otimes \chi_P^{\varepsilon_P/2}$. For $\GL_n$ this is equivalent to replacing 
$\pi$ by $\pi \otimes |\ |^{\varepsilon_n/2}$, where $\varepsilon_n \in \{0,1\}$ and $\varepsilon_n \equiv n-1 \pmod{2}.$
Any point of evaluation of a global $L$-function attached to an algebraic data (think of a motivic $L$-function) 
should be an integer for the $L$-value to be critical in the sense of Deligne \cite{deligne}. 
This motivates the following definition which is independent of $\pi$ and depends only on $(G,P)$.

\begin{defn}
Let $G$ be a connected reductive $p$-adic group and $P$ a maximal parabolic subgroup. We say that $P$ is critical for $G$ if the point 
of evaluation $k = - \langle \rho_P, \alpha_P \rangle$ satisfies the condition:
$$
jk \in h_j + \Z, \quad \forall \ 1 \leq j \leq m.
$$
\end{defn}

\medskip

\subsection{Hypothesis on local critical $L$-values}

We can now formulate the arithmeticity hypothesis for local critical $L$-values. 

\begin{hyp}
\label{hyp:local-L-value}
Let $G$ be a connected reductive $p$-adic group and $P$ a maximal parabolic subgroup. Assume that $P$ satisfies the following two conditions:
\begin{enumerate}
\item[(i)]  the integrality condition: $\rho_P|_{A_P} \in X^*(A_P)$; and 
\smallskip
\item[(ii)] the criticality condition: $P$ is critical for $G$.  
\end{enumerate} 
Let $\sigma$ be a smooth, absolutely-irreducible, half-integrally unitary, essentially-tempered, admissible, generic representation of $M_P$ over 
a field of coefficients $E$.  
Let $k = -\langle \rho_P, \alpha_P\rangle$ be the point of evaluation, and let $s_j  \in \{jk, \, jk+1\}$ for any $1 \leq j \leq m$. 
Then for 
any embedding $\iota: E \to \C$ we have: 
\begin{enumerate}
\item $L(s_j, {}^\iota\sigma, \, \tilde r_j) \in \iota(E),$ and furthermore  
\smallskip
\item for any $\tau \in \Gal(\bar\Q/\Q)$ we have $\tau(L(s_j, {}^\iota\sigma, \, \tilde r_j)) \ = \ L(s_j, {}^{\tau \circ\iota}\sigma, \, \tilde r_j).$
\end{enumerate} 
\end{hyp}

As already mentioned, this hypothesis can be verified in some various concrete examples of interest. We briefly mention two examples below; these contexts 
are amplified in Sect.\,\ref{sec:examples}; the reader will readily appreciate that such examples may be generalised. 

\begin{exam}[Local $L$-functions for $\GL_n(F)$]
\label{exam-GLn-arith-L-value}
{\rm 
If $\pi$ is an irreducible admissible representation of $\GL_n(F)$ then it follows from Clozel \cite[Lem.\,4.6]{clozel} that for any $k_0 \in \Z$ and 
any $\tau \in \Aut(\C)$ one has:  
$$
\tau\left(L(k_0 + \tfrac{1-n}{2}, \pi)\right) \ = \ L(k_0 + \tfrac{1-n}{2}, {}^\tau\pi). 
$$
The reader can check that Hyp.\,\ref{hyp:local-L-value} follows from this Galois equivariance after appealing to the details in 
Sect.\,\ref{sec:example-rankin-selberg}. 
Such a Galois equivariance can be reformulated as  
$$
\tau\left(L(k_0, \pi)\right) \ = \ L(k_0, {}^\tau\pi \otimes \varepsilon_\tau^{n-1}), 
$$
which is useful in other situations; see the next example below. 
}\end{exam}

\begin{exam}[Local $L$-functions for $\rO_{2n}(F)$]
\label{exam-O2n-arith-L-value}
{\rm 
Suppose $\pi$ is an irreducible tempered representation of $\rO_{2n}(F)$ as in the proof of Hyp.\,\ref{hyp:arith-temp} for orthogonal groups; in particular,
$\pi \hookrightarrow \theta_1 \times \cdots \times \theta_q \rtimes \pi_0$, with notations as therein. Then,     
the $L$-parameters are related as: 
$\phi_\pi \ = \ \phi_{\theta_1} + \cdots + \phi_{\theta_1} + \phi_{\pi_0} + \phi_{\theta_1}^{\sf v} + \cdots + \phi_{\theta_q}^{\sf v}$ 
(see Atobe--Gan \cite[Desideratum 3.9]{atobe-gan}). In particular, 
for $L$-functions, evaluating at $s = k \in \Z$ (see Sect.\,\ref{sec:example-orthogonal} for the fact that the point of evaluation is an integer), we get: 
$
L(k, \pi) \ = \ L(k, \theta_1)\cdots L(k, \theta_q) \cdot L(k, \pi_0) \cdot L(k, \theta_1^{\sf v})\cdots L(k, \theta_q^{\sf v}).
$
Apply $\tau \in \Aut(\C)$ to both sides 
while using Example \ref{exam-GLn-arith-L-value} to get $\tau(L(k, \pi))$ is 
$$
L(k, {}^\tau\theta_1 \otimes \varepsilon_\tau^{a_1-1})\cdots L(k, {}^\tau\theta_q\otimes \varepsilon_\tau^{a_q-1}) \cdot \tau(L(k, \pi_0)) \cdot
L(k, {}^\tau\theta_1^{\sf v} \otimes \varepsilon_\tau^{a_1-1})\cdots L(k, {}^\tau\theta_q^{\sf v}\otimes \varepsilon_\tau^{a_q-1}). 
$$
Now use \eqref{eqn:tau-pi-O-2n} for ${}^\tau\pi$; then take its $L$-function evaluated at $s = k$ to get: 
$$
\tau(L(k, \pi)) \ = \ L(k, {}^\tau\pi), 
$$ 
assuming by induction that $\tau(L(k, \pi_0)) \ = \ L(k, {}^\tau\pi_0)$ holds for a discrete series representation $\pi_0$ of $\rO_{2n_0}(F)$. The proof for 
a discrete series representation follows the same reduction strategy as in the proof of Hyp.\,\ref{hyp:arith-temp} for even orthogonal groups.
}\end{exam}

I expect that a proof of Hyp.\,\ref{hyp:local-L-value} in the general case should come from an arithmetic-analysis of Shahidi's theory of local factors.

\medskip
\subsection{An arithmetic variant of Thm.\,\ref{thm:waldy-rational} for normalised intertwining operator}

We can now strengthen Thm.\,\ref{thm:waldy-arithmetic} for the normalised intertwining operator. 

\begin{thm}
\label{thm:waldy-arithmetic-normalised}
Let the notations and hypotheses be as in Thm.\,\ref{thm:waldy-arithmetic}. Assume furthermore that Hyp.\,\ref{hyp:local-L-value} holds.  
Let $T_{\rm norm} = T_{\rm norm}(s, {}^\iota\sigma)|_{s = k}$ be the normalised standard intertwining operator (see \eqref{eqn:T-norm}) 
at the point 
of evaluation $s = k$. Then there exists an $E$-linear $G$-equivariant map
$$
T_{{\rm norm}, E} : \aInd_P^G(\sigma) \ \longrightarrow \ \aInd_Q^G(\sigma \otimes \delta_Q)
$$ 
such that for any embedding $\iota : E \to \C$ we have:
$$
T_{{\rm norm}, E} \otimes_{E,\iota} 1_\C  \ = \ T_{\rm norm}. 
$$

\end{thm}

\bigskip
\section{Examples} 
\label{sec:examples}

\medskip
\subsection{Rankin--Selberg $L$-functions}
\label{sec:example-rankin-selberg}
(See case ($A_{N-1}$) in \cite[Appendix B]{shahidi-book}.)
$$
\xymatrix{
\stackrel{\alpha_1}{\bullet} \ar@{-}[r] & {\cdots} & \stackrel{\alpha_{n-1}}{\bullet} \ar@{-}[l] \ar@{-}[r] & \circ & 
\stackrel{\alpha_{n}}{\bullet} \ar@{-}[l] \ar@{-}[r]& {\cdots} & \stackrel{\alpha_{n+n'-2}}{\bullet} \ar@{-}[l]
}$$

\medskip
\begin{enumerate}
\item Ambient group: $G = \GL(N)/F$ with $N \geq 2$. 
\smallskip
\item Maximal parabolic subgroup: take $N = n+n'$ and let $P$ be the maximal parabolic subgroup with Levi $M_P = \GL(n) \times \GL(n')$; 
the deleted simple root $\alpha_P = e_n-e_{n+1}.$ 
\smallskip
\item The integrality condition $\rho_P \in X^*(A_P)$ holds if and only if $nn' \equiv 0 \! \pmod{2}.$ \\
The set of roots with roots spaces appearing in the Lie algebra $\mathfrak{n}_P$ of the unipotent radical of $P$ is 
$\{e_i - e_j : 1 \leq i \leq n, \ n+1 \leq j \leq n+n'\}.$
Hence $
\rho_P = \frac{n'}{2}(e_1+\dots + e_n) - \frac{n}{2}(e_{n+1}+\dots+e_{n+n'}).$
Whence, $\rho_P$ as a character of $A_P$ is given by: ${\rm diag}(t1_n, t'1_{n'}) \mapsto (t t')^{nn'/2},$ which is integral if and only if $nn'$ is even. 
It is curious that this very condition was imposed in \cite{harder-raghuram-book} due to motivic considerations (the tensor product motive therein needed to be even rank). 

\smallskip
\item At the level of dual groups, $\LM_P^\circ \cong \GL_n(\C) \times \GL_{n'}(\C)$ acts irreducibly on the Lie algebra 
${}^L\mathfrak{n}_P \cong M_{n \times n'}(\C)$ of $\LN_P$; $m=1$. 

\smallskip
\item Inducing data consists of $\pi$ and $\pi'$ which are essentially tempered irreducible generic 
representations of $\GL_n(F)$ and $\GL_{n'}(F)$ then $L(s, \pi \otimes \pi', \tilde r_1) = L(s, \pi \times \pi^{\sf v})$ is the local 
Rankin--Selberg $L$-function attached to $\GL_n(F) \times \GL_{n'}(F).$ 

\smallskip
\item The point of evaluation is $k = - \langle \rho_P, \alpha_P \rangle = -N/2$. 

\smallskip
\item $P$ is critical for $G$: $M = M_1M_2$ with $M_1 = \GL_n(F)$ and $M_2 = \GL_{n'}(F);$
$\varepsilon_{M_1} = \varepsilon_n$ and $\varepsilon_{M_2} = \varepsilon_{n'}$ (recall: 
$\varepsilon_n \in \{0,1\}$ by $\varepsilon_n \equiv n-1 \pmod{2}$);  
$h_1 = (\varepsilon_n - \varepsilon_{n'})/2$; and $k \in h_1 + \Z$. 
\end{enumerate}

\subsection{$L$-functions for orthogonal groups}
\label{sec:example-orthogonal}
See case ($D_{n,i}$) in \cite[Appendix A]{shahidi-book}; the corresponding global context is studied in \cite{bhagwat-raghuram}. 
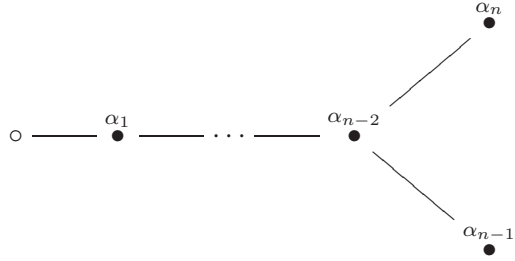
$$
\xymatrix{
& & & & \stackrel{\alpha_{n}}{\bullet} \\
{\circ} \ar@{-}[r] & \stackrel{\alpha_1}{\bullet} \ar@{-}[r] &  {\cdots} &   \stackrel{\alpha_{n-2}}{\bullet} \ar@{-}[l] \ar@{-}[ru] \ar@{-}[rd] & \\ 
& & & & \stackrel{\alpha_{n-1}}{\bullet}
}$$

\medskip
\begin{enumerate}
\item Ambient group $G = \rO(n+1, n+1) = \{g \in \GL_{2n+2}(F) : {}^tg \cdot J_{2n+2} \cdot g = J_{2n+2}\}$, where 
$J_{2n+2}(i,j) = \delta(i, 2n+3-j)$; this is the split even orthogonal group of rank $n+1$; the maximal torus consists of all diagonal
matrices ${\rm diag}(t_0,t_1,\dots,t_n, t_n^{-1},\dots, t_1^{-1}, t_0^{-1}).$

\smallskip
\item Let $P$ be the maximal parabolic subgroup described by the above Dynkin diagram; deleted simple root $\alpha_P = e_0-e_1;$
Levi subgroup is 
$$
M_P =  \left\{ m_{t,h} = \begin{pmatrix}
t & & \\ & h & \\ & & t^{-1} \end{pmatrix} : t \in \GL_1(F), \  h \in \rO(n,n)\right\}; 
$$
clearly, $A_P = \{m_{t,1} \in M_P : t \in \GL_1\}$; the unipotent radical of $P$ is:
$$
N_P = \left\{ u_{y_1, y_2, \ldots, y_{2n}} = 
\left( \begin{array}{ccccccc} 
  1 & y_1 & y_2 & \ldots & y_{2n} & 0 \\
                          & 1  &  & &  &-y_{2n} \\ 
                          & & 1  & & & -y_{2n-1} \\  
                          & &   &   \ddots & & \vdots \ \\
                          & &  & &  \ddots  &  -y_{1} \\
                          & &   & & & 1
                          \end{array}\right) \ \ | \ \  
                          y_1,\dots,y_{2n} \in F \right\}. 
$$ 

\smallskip
\item The integrality condition on $\rho_P$ holds for all $n.$ 
The set of roots with root spaces in the Lie algebra of $N_P$ is 
$
\{e_0-e_1, \, e_0-e_2,\dots, \, e_0-e_{2n}\}. 
$
Hence 
$$
\rho_P = n e_0 - \tfrac12(e_1+e_3+\dots + e_{2n});
$$ from the maximal torus one has 
$e_{n+1} = -e_n, e_{n+2} = -e_{n-1}, \dots, e_{2n} = -e_1$ from which it follows that $\rho_P = n e_0$. Whence, 
$\rho_P|_{A_P}$ is the integral character $t = m_{t,1} \mapsto t^n.$ 

\smallskip
\item At the level of dual groups, $\LM_P^\circ = \left\{ m_{t,h} : t \in \C^*, \  h \in \rO(n,n)(\C)\right\}$ acts irreducibly on the Lie algebra 
${}^L\mathfrak{n}_P$ of $\LN_P^\circ$; $m=1$ and $r_1$ is the standard representation of $\rO(n,n)(\C)$ twisted by the $\C^*$ in the obvious way. 

\smallskip
\item Inducing data is of the form $\chi \otimes \pi$ for a character $\chi: F^* \to \C^*$, and a tempered, irreducible, 
generic representation 
$\pi$ of $\rO(n,n)(F).$  The local $L$-function $L(s, \chi \otimes \pi, \tilde r_1)$ is the local Rankin-Selberg $L$-function 
$L(s, \chi \otimes \tilde r_1(\pi))$ for $\GL_1 \times \GL_{2n}.$   

\smallskip
\item The point of evaluation is $k = - \langle \rho_P, \alpha_P \rangle = -n.$ 

\smallskip
\item $P$ is critical for $G$, since $M = M_1M_2$ with $M_1 = \GL_1(F)$ and $M_2 = \rO(n,n)(F);$ $\rho_{M_i}$ is integral; 
$h_1 = 0$; $k \in \Z$. 
\end{enumerate}

\subsection{Exterior square $L$-functions}
\label{sec:example-exterior-square}
(See case ($C_{n-1,ii}$) in \cite[Appendix A]{shahidi-book}.)
$$
\xymatrix{
\stackrel{\alpha_1}{\bullet} \ar@{-}[r] & \stackrel{\alpha_2}{\bullet} \ar@{-}[r] & {\cdots} &   \stackrel{\alpha_{n-1}}{\bullet} \ar@{-}[l]   \ar@2{-}[r] |-{\SelectTips{cm}{}\object@{<}} & {\circ}
}$$
\medskip
\begin{enumerate}
\item Ambient group 
$G = \Sp_{2n}(F) =  \left\{ g \in \GL_{2n}(F) : {}^tg \cdot \begin{pmatrix}  & J_n \\ -J_n&  \end{pmatrix} \cdot g = \begin{pmatrix}  & J_n \\ -J_n&  \end{pmatrix} \right\},$ 
where $J_n(i,j) = \delta(i, r-j+1)$ and ${}^tg$ is the transpose of $g$. 

\smallskip
\item Maximal parabolic subgroup as depicted by the above Dynkin diagram has Levi subgroup:
$M_P = \left\{ \begin{pmatrix} h & \\ & {}^{(t)}h^{-1} \end{pmatrix} : h \in \GL_n(F) \right\}$ where 
${}^{(t)}h = J_n \cdot {}^t h \cdot J_n$ is the `other-transpose' of $h$ defined by $({}^{(t)}h)_{i,j} = h_{n-j+1, n-i+1}.$ The deleted simple root $\alpha_P = 2 e_n.$ 
 
\smallskip 
\item The integrality condition $\rho_P \in X^*(A_P)$ holds if and only if $n \equiv 0,3 \pmod{4}.$ \\
The Lie algebra of the unipotent radical of $P$ is of the form 
$$
\mathfrak{n}_P \ = \ 
\left\{ \begin{pmatrix} 0_n & X \\ 0_n & 0_n \end{pmatrix} : X \in M_n(F), \ {}^{(t)}X = X \right\}.
$$
The set of roots with root spaces appearing in $\mathfrak{n}_P$ is
$
\{e_1-e_{n+1},\ \dots,\ e_1-e_{2n}, \ e_2-e_{n+1},\ \dots,\ e_2-e_{2n-1}, \ \dots, 
 e_n-e_{n+1}\}.
$
Keeping in mind that $e_j = -e_{2n-j+1}$ we get
$$
\rho_P = \frac{n+1}{2}(e_1+ \cdots + e_n).
$$ 
Whence, $\rho_P|_{A_P}$ is given by: ${\rm diag}(t1_n, t^{-1}1_n) \mapsto t^{n(n+1)/2},$ which is integral if and only if 
$n(n+1)/2$ is even, i.e., $n \equiv 0 \ {\rm or} \ 3 \! \pmod{4}.$
 
\smallskip
\item Dual groups: 
$\LG^\circ = \SO(2n+1,\C) = \left\{ g \in \SL_{2n+1}(F) : {}^tg \cdot J_{2n+1} \cdot g = J_{2n+1} \right\},$ \\
$$\LM_P^\circ = \left\{ 
m_g = \begin{pmatrix} g & & \\  & 1 & \\ & & {}^{(t)}g^{-1} \end{pmatrix} : g \in \GL_n(\C) \right\} \cong \GL_n(\C);
$$  
{\small
$$
{}^L\mathfrak{n}_P \ = \ \left\{ 
n_{y,X} = 
\begin{pmatrix}
0_n & y & X \\
0_{1 \times n} & 0 & -{}^ty J_n  \\
0_n & 0_{n \times 1} & 0_n
\end{pmatrix}: \ y \in M_{n \times 1}(\C), \ X \in M_{n \times n}(\C), \ {}^{(t)}X = -X \right\}. 
$$} 

\noindent
The adjoint action of $\LM_P^\circ$ on ${}^L\mathfrak{n}_P$ is the direct sum of two irreducible representations with representation spaces 
$V_1 = \{n_{y,0} \in {}^L\mathfrak{n}_P\}$ and $V_2 = \{n_{0,X} \in {}^L\mathfrak{n}_P\}$ of dimensions $n$ and $n(n-1)/2$, respectively; 
$r_1$ is the standard representation and $r_2$ is 
the exterior square representation; $m=2$. 
The centre $\LA_P^\circ$ of $\LM_P^\circ$ consists of elements $a_t = m_{t\cdot I_n}$ for $t \in \C^\times$; then $a_t$ acts on $V_1$ by the scalar $t$ and on $V_2$ by the scalar $t^2$. 
\smallskip
\item The inducing data is a half-integrally unitary, irreducible, essentially-tempered, generic representation $\pi$ of $\GL_n(F);$ for the $L$-functions we have:  
   \begin{enumerate}
   \item $L(s, \pi, \tilde r_1) = L(s, \pi),$ the standard $L$-function for $\GL(n)$, and 
   \item $L(s, \pi, \tilde r_2) = L(s, \pi, \wedge^2),$ the exterior square $L$-function for $\GL(n).$
   \end{enumerate}
   
\smallskip
\item 
The point of evaluation is $k = - \langle \rho_P, \alpha_P \rangle = - \frac{n+1}{2}.$  

\smallskip
\item 
$P$ is critical for $G$. Since $\varepsilon_M = \varepsilon_n,$ $h_1 = \varepsilon_n/2,$ and $h_2 = \varepsilon_n$; 
hence $jk \in h_j + \Z$ holds for $j = 1, 2.$
\end{enumerate}

\subsection{Explicit intertwining calculation for the case of $\GL(2)$}

Some essential features of main results are already visible for the example of $\GL(2)$ from first principles; 
although the reader is warned of the well-known dictum that $\GL(2)$ is  misleadingly simple and it is difficult to carry out a straightforward generalisation 
of such calculations. 

\medskip 
Let $E/\Q$ be a finite extension,  
and for $i = 1,2$, let $\chi_i : F^\times \to E^\times$ be a smooth character, and $\chi_i^\circ$ its restriction to $\cO_F^\times.$ 
Let $\iota : E \to \C$ be an embedding of fields, and 
${}^\iota\chi_i = \iota \circ \chi_i$ be the corresponding $\C$-valued character of $F^\times.$ Let $G = \GL_2(F)$, $K = \GL_2(\cO_F),$ and for 
$m \geq 0$ let $K(m)$ be the principal congruence subgroup of $K$ of level $m$; $K(0) = K.$ The standard intertwining operator $T_{\rm st}(s)$ at the 
point of evaluation $s = -1$ between the $K(m)$-invariants of algebraically induced representations has the shape: 
$$
T_{\rm st}(s)|_{s = -1} \ : \ \aInd_B^G({}^\iota\chi_1 \otimes {}^\iota\chi_2)^{K(m)}
 \ \longrightarrow \ 
\aInd_B^G({}^\iota\chi_2(1) \otimes {}^\iota\chi_1(-1))^{K(m)}.
$$
The standing assumptions that ${}^\iota\chi_i$ is half-integrally unitary, essentially tempered, and $\pi = {}^\iota\chi_1 \otimes {}^\iota\chi_2$ is on the right of the unitary axis with respect to $G$ implies that $T := T_{\rm st}(s)|_{s = -1}$ is finite. A function in $\aInd_B^G({}^\iota\chi_1 \otimes {}^\iota\chi_2)^{K(m)}$ is completely 
determined by its restriction to $K$. This gives us the following diagram: 
$$
\xymatrix{
\aInd_B^G({}^\iota\chi_1 \otimes {}^\iota\chi_2)^{K(m)} \ar[d]_{f \mapsto f|_K} \ar[rr]^{T}
& & 
\aInd_B^G({}^\iota\chi_2(1) \otimes {}^\iota\chi_1(-1))^{K(m)}\ar[d]^{f \mapsto f|_K} \\
\aInd_{K \cap B}^K({}^\iota\chi_1^\circ \otimes {}^\iota\chi_2^\circ)^{K(m)} \ar[rr]^{T^\circ}
& & 
\aInd_{K \cap B}^K({}^\iota\chi_2^\circ \otimes {}^\iota\chi_1^\circ)^{K(m)} 
}$$
Working with $K(m)$-invariants is not strictly necessary; it has the virtue of making the spaces finite-dimensional and $G$-action is replaced by action of the Hecke-algebra 
$\cC^\infty_c(G/\!/K(m)).$
Let $f^\circ \mapsto \tilde{f^\circ}$ denote the inverse of $f \mapsto f|_K.$ Let $f \in \aInd_{B}^G({}^\iota\chi_1 \otimes {}^\iota\chi_2)^{K(m)}$
and for brevity let $f^\circ = f|_K.$ Since $T(f)$ is determined by its restriction to $K$, we have:
$$
T^\circ(f^\circ)(k) \ = \ T(f)(k) \ = \ 
\int_F f(
\begin{pmatrix} & -1\\ 1 & \end{pmatrix} \begin{pmatrix} 1 & x \\  & 1\end{pmatrix}k) \, dx, \quad k \in K.
$$
Break up the integral over $x \in \cP^{-m}$ and $x \notin \cP^{-m}.$ Note that 
$$
\int_{\cP^{-m}} f(
\begin{pmatrix} & -1\\ 1 & \end{pmatrix} \begin{pmatrix} 1 & x \\  & 1\end{pmatrix}k) \, dx \ = \ 
\sum_{a \in \cP^{-m}/\cP^m} 
\int_{y \in \cP^m} f(
\begin{pmatrix} & -1\\ 1 & \end{pmatrix} \begin{pmatrix} 1 & a + y \\  & 1\end{pmatrix} k) \, dy.  
$$
We write 
$$
\begin{pmatrix} 1 & a + y \\  & 1\end{pmatrix} k \ = \ 
\begin{pmatrix} 1 & a \\  & 1\end{pmatrix} \begin{pmatrix} 1 & y \\  & 1\end{pmatrix} k \ = \ 
\begin{pmatrix} 1 & a \\  & 1\end{pmatrix} k \cdot k^{-1} \begin{pmatrix} 1 & y \\  & 1\end{pmatrix} k
$$
and use that $K(m)$ is a normal subgroup of $K$ and $\tilde{f^\circ}$ is right $K(m)$-invariant to get
\begin{equation}
\label{eqn:x-in-p-m}
\int_{\cP^{-m}} f(
\begin{pmatrix} & -1\\ 1 & \end{pmatrix} \begin{pmatrix} 1 & x \\  & 1\end{pmatrix}k) \, dx \ = \ 
{\rm vol}(\cP^m) 
\sum_{a \in \cP^{-m}/\cP^m} f(
\begin{pmatrix} & -1\\ 1 & \end{pmatrix} \begin{pmatrix} 1 & a \\  & 1\end{pmatrix} k),
\end{equation}
which is a finite-sum. For the integral over $x \notin \cP^{-m}$ use: 
$$
\begin{pmatrix} & -1\\ 1 & \end{pmatrix} \begin{pmatrix} 1 & x \\  & 1\end{pmatrix} \ = \ 
\begin{pmatrix}x^{-1} & \\  & x\end{pmatrix} 
\begin{pmatrix}1 & -x\\  & 1\end{pmatrix} 
\begin{pmatrix}1 & \\ x^{-1} & 1\end{pmatrix}; 
$$
break up $\int_{x \notin \cP^{-m}}$ as $\sum_{r = m}^{\infty} \int_{\varpi^{-r}\cO^\times}$ to get  
$$
\int_{x \notin \cP^{-m}} f(
\begin{pmatrix} & -1\\ 1 & \end{pmatrix} \begin{pmatrix} 1 & x \\  & 1\end{pmatrix}k) \, dx \ = \ 
\sum_{r = m}^\infty \int_{\varpi^{-r}\cO^\times}  
f(
\begin{pmatrix}x^{-1} & \\  & x\end{pmatrix} 
\begin{pmatrix}1 & -x\\  & 1\end{pmatrix} 
\begin{pmatrix}1 & \\ x^{-1} & 1\end{pmatrix}k) \, dx. 
$$
Since $x^{-1} \in \cP^m$, using the equivariance of $f$, the right hand side simplifies to: 
$$
\sum_{r = m}^\infty \int_{\varpi^{-r}\cO^\times}   {}^\iota \chi_1(x^{-1}){}^\iota\chi_2(x) f(k) \, dx. 
$$
Make the substitution $x = \varpi^{-r}u$ with $u \in \cO^\times$; then $dx = q^r du = q^r d^\times u$, and one gets: 
$$
f(k) 
\sum_{r = m}^\infty 
{}^\iota \chi_1(\varpi^{r}){}^\iota\chi_2(\varpi^{-r}) q^r
\int_{\cO^\times}   {}^\iota \chi_1(u^{-1}){}^\iota\chi_2(u)\, d^\times u. 
$$
The inner integral is nonzero if and only if ${}^\iota \chi_1(u) = {}^\iota \chi_2(u)$ for all $u \in \cO^\times$; assuming this to be the case we get: 
\begin{multline}
\label{eqn:x-not-in-p-m}
\int_{x \notin \cP^{-m}} f(
\begin{pmatrix} & -1\\ 1 & \end{pmatrix} \begin{pmatrix} 1 & x \\  & 1\end{pmatrix}k) \, dx \ = \\ 
{\rm vol}(\cO^\times) \cdot {}^\iota \chi_1(\varpi^{m}){}^\iota\chi_2(\varpi^{-m}) q^m \cdot 
\left(1 - {}^\iota \chi_1(\varpi){}^\iota\chi_2(\varpi^{-1}) q\right)^{-1} \cdot f(k). 
\end{multline}
For $G = \GL(2)$, the point of evaluation $k = -1$, and $\left(1 - {}^\iota \chi_1(\varpi){}^\iota\chi_2(\varpi^{-1}) q\right)^{-1} $ is nothing but 
$L(s, {}^\iota \chi_1 \otimes {}^\iota\chi_2, \tilde{r}) = L(s, {}^\iota \chi_1 \otimes {}^\iota\chi_2^{-1})$ evaluated at this point of evaluation. 
Putting \eqref{eqn:x-in-p-m} and \eqref{eqn:x-not-in-p-m} together, one sees that $T(f)(k)$ is a finite-sum: 
\begin{multline}
\label{eqn:T-finite-sum}
{\rm vol}(\cP^m) 
\sum_{a \in \cP^{-m}/\cP^m} f(
\begin{pmatrix} & -1\\ 1 & \end{pmatrix} \begin{pmatrix} 1 & a \\  & 1\end{pmatrix} k) \ + \\ 
\delta(\chi_1^\circ, \chi_2^\circ) \cdot 
{\rm vol}(\cO^\times) \cdot {}^\iota \chi_1(\varpi^{m}){}^\iota\chi_2(\varpi^{-m}) q^m \cdot 
L(-1, {}^\iota \chi_1 \otimes {}^\iota\chi_2^{-1}) \cdot f(k). 
\end{multline}

\medskip

For brevity, let $\fI = \aInd_B^G({}^\iota\chi_1 \otimes {}^\iota\chi_2)^{K(m)}$ and $\tilde\fI = \aInd_B^G({}^\iota\chi_2(1) \otimes {}^\iota\chi_1(-1))^{K(m)}$; 
these induced representations admit a natural $E$-structures; define
$$
\fI_0 \ := \ \aInd_B^G(\chi_1 \otimes \chi_2)^{K(m)}, \quad \tilde\fI_0 \ := \ \aInd_B^G(\chi_2(1) \otimes \chi_1(-1))^{K(m)}. 
$$
If $\chi : F^\times \to E^\times$ is a locally constant homomorphism then for any integer $n$ we denote $\chi(n) = \chi\otimes |\ |^n$ the $E$-valued character: $u \mapsto \chi(u)$ for 
all $u \in \cO^\times$ and $\varpi \mapsto q^{-n}\chi(\varpi)$. 
It is clear then that $\fI = \fI_0 \otimes_{E, \iota} \C$ and $\tilde\fI = \tilde\fI_0 \otimes_{E, \iota} \C.$ Note that $\fI_0$ consists of all $E$-valued functions in 
$\fI$; similarly, $\tilde\fI_0.$ The local $L$-value that appears in \eqref{eqn:T-finite-sum} is $E$-rational, i.e., 
$L(-1, {}^\iota \chi_1 \otimes {}^\iota\chi_2^{-1}) = \left(1 - {}^\iota \chi_1(\varpi){}^\iota\chi_2(\varpi^{-1}) q\right)^{-1}  \in \iota(E),$
and furthermore, if $L_0(-1, \chi_1 \otimes \chi_2^{-1}) = \left(1 - \chi_1(\varpi)\chi_2(\varpi^{-1}) q\right)^{-1} \in E$ then 
$\iota(L_0(-1, \chi_1 \otimes \chi_2^{-1})) = L(-1, {}^\iota \chi_1 \otimes {}^\iota\chi_2^{-1}).$ 
It is clear now from \eqref{eqn:T-finite-sum} that $T(\fI_0) \subset \tilde\fI_0;$ also that if $T_0 = T|_{\fI_0}$ then $T = T_0 \otimes_{E, \iota} \C.$

\bigskip
\bigskip
\bibliographystyle{plain}

\bigskip

\bigskip

\end{document}